\documentclass{article}
\usepackage[utf8]{inputenc}
\usepackage{hyperref}
\usepackage{amsmath, amsthm, amssymb, latexsym,epsfig,amsthm,enumerate,multicol,wasysym}
\usepackage{xcolor, comment}
\usepackage{bbm}
\usepackage{mathtools}
\usepackage{enumitem}
\usepackage[pagewise]{lineno}

\numberwithin{equation}{section}

\title{A Bounded Regret Strategy for Linear Dynamics with Unknown Control}
\author{Jacob Carruth}
\date{November 2023}

\newtheorem{thm}{Theorem}
\newtheorem{lem}{Lemma}
\newtheorem{rmk}{Remark}

\newcommand{\R}{\mathbb{R}}
\newcommand{\prob}{\text{Prob}}
\newcommand{\eprob}{\emph{Prob}}
\newcommand{\E}{\text{E}}
\newcommand{\var}{\text{Var}}
\newcommand{\cF}{\mathcal{F}}
\newcommand{\cost}{\textsc{Cost}}
\newcommand{\ecost}{\textsc{ECost}}

\newcommand{\br}{\text{BR}}
\newcommand{\ebr}{\emph{BR}}

\newcommand{\titau}{\tilde{\tau}}
\newcommand{\tit}{\tilde{t}}

\newcommand{\httau}{\hat{\tau}}
\newcommand{\htt}{\hat{t}}

\newcommand{\cE}{\mathcal{E}}
\newcommand{\op}{\text{opt}}
\newcommand{\eop}{\emph{opt}}

\newcommand{\ecg}{\emph{CG}}
\newcommand{\ii}{\text{ii}}
\newcommand{\iii}{\text{iii}}
\newcommand{\I}{\text{i}}

\newcommand{\mreg}{\text{MReg}}

\pdfoutput=1

\begin{document}

\maketitle

\begin{abstract}
    We consider a simple linear control problem in which a single parameter $b$, describing the effect of the control variable, is unknown and must be learned. We work in the setting of agnostic control: we allow $b$ to be any real number and we do not assume that we have a prior belief about $b$. For any fixed time horizon, we produce a strategy whose expected cost is within a constant factor of the best possible.
\end{abstract}

\section{Introduction}

In this paper, we consider a simple linear control problem with a single unknown parameter $b$. We do not assume that we are given a prior belief about $b$, and we seek to bound a quantity called the regret, which compares our expected cost to that of an opponent with perfect knowledge of $b$. 

This falls within the purview of adaptive control. In the adaptive control literature, one typically seeks a control strategy that minimizes the regret as the time horizon tends to infinity. Here, we explore a variant of adaptive control called agnostic control, in which one seeks to minimize the regret over a fixed, potentially short, time horizon; see also \cite{almostoptimal2023, carruth2022controlling, boundeda, fefferman2021optimal, gurevich2022optimal}.

Let $q(t)$ denote the position of a particle at time $t$. Given some starting position $q_0 \in \R$, the particle evolves according to the stochastic ODE
\begin{equation}\label{eq: int 1}
dq(t) = (aq(t) + bu(t))dt + dW(t), \qquad q(0) = q_0.
\end{equation}
Here $a$ and $b$ are real numbers, $W(t)$ is a standard Brownian motion (here ``standard'' means that $W(0)=0$ and the variance of $W(t)$ is $t$), and $u$ is the control variable. We treat the parameter $a$ as a fixed, known quantity, but we treat $b$ as an unknown parameter about which we have no prior belief. We choose $u(t)$ depending on the history up to time $t$, i.e., $(q(\tau))_{0\le\tau\le t}$. The control $u(t)$ is allowed to depend on $a$, but not on the unknown parameter $b$. 

More precisely, a \emph{strategy} is a choice of $u(t)$ for every $t$ and every history $(q(\tau))_{0 \le \tau \le t}$, subject to certain technical conditions that allow us to make sense of solutions to \eqref{eq: int 1} for arbitrary $b \in \R$ (a precise definition is given in Section \ref{sec: strategies}). For a given strategy $\sigma$, we let $q^\sigma(t,b)$ and $u^\sigma(t,b)$ denote the corresponding solutions to \eqref{eq: int 1} (note that these are random variables). Fix a time horizon $T>0$. We define the cost of $\sigma$ to be the random variable
\[
\cost(\sigma,b) = \int_0^T \big((q^\sigma(t,b))^2 + (u^\sigma(t,b))^2\big)dt,
\]
and we write $\ecost(\sigma,b)$ for the expected value of $\cost(\sigma,b)$. 

For a given value of $b$, there exists a strategy $\sigma_\op(b)$ that minimizes the quantity $\ecost(\sigma,b)$ over all strategies $\sigma$ (this is well-known, see, e.g., \cite{astrom}). We write 
\[
\ecost_\op(b) = \ecost(\sigma_\op(b),b).
\]
It is helpful to think of $\sigma_\op(b)$ as the strategy of an opponent who knows the value of $b$ and plays optimally---this leads to the notion of \emph{regret}.

We define the (multiplicative) \emph{regret}\footnote{We will use the multiplicative regret, often referred to as the competitive ratio. There are other notions of regret; see \cite{almostoptimal2023, boundeda} for a discussion of two others.} of a strategy $\sigma$ by 
\[
\mreg(\sigma,b) = \frac{\ecost(\sigma,b)}{\ecost_\op(b)},
\]
and the \emph{worst-case regret} of $\sigma$ by
\[
\mreg^*(\sigma) = \sup_{b \in \R} \mreg(\sigma, b).
\]
We seek a strategy that minimizes the \emph{worst-case regret}, and we refer to such a strategy as an \emph{optimal agnostic strategy}.

In \cite{almostoptimal2023, boundeda}, we produced almost optimal agnostic strategies for the dynamics \eqref{eq: int 1} in the case in which the parameter $a$ is unknown and $b$ is known (for simplicity, in those papers we took $b=1$). The problem considered here, in which $b$ is unknown, appears to be more difficult. This is because it has the following feature: If we set $u(t) = (\Delta t)^{-1/2}$ for a very short time increment $\Delta t$, then we learn a lot about the parameter $b$ and only incur a finite cost. Therefore, we may want to take $u = \pm \infty$; this had better be done carefully.

In this paper, we tackle the more modest problem of exhibiting a strategy with \emph{bounded} multiplicative regret. Specifically, we prove the following theorem

\begin{thm}\label{thm: main}
    Fix $|q_0| \ge 1$ and $T>0$ as above. Then there exists a constant $C>0$ such that for any $a \in \R$ the strategy $\ebr$ defined in Sections \ref{sec: conventions}--\ref{sec: neg} satisfies
    \[
    \ecost(\ebr,b) \le C\cdot \ecost_\eop(b)\;\text{for any}\; b \in \R.
    \]
\end{thm}
\noindent Clearly, Theorem \ref{thm: main} implies that
\[
\mreg^*(\br) \le C.
\]
We emphasize that the constant $C$ in Theorem \ref{thm: main} depends on $q_0$ and $T$ but is independent of the parameter $a$.

We note that our assumption that $|q_0| \ge 1$ in Theorem \ref{thm: main} can be replaced by $|q_0| \ge q_{\text{small}}$ for any $q_{\text{small}}>0$. Unfortunately, however, the resulting constant $C$ in Theorem \ref{thm: main} then tends to infinity as $q_{\text{small}}$ tends to zero. We think that with a bit of care this can be avoided, and we are currently working on producing a bounded regret strategy for small $q_0$.

We now give an overview of some of the ideas behind the strategy $\br$. We will assume that $q_0 \ge 1$. By symmetry, this also defines the strategy for $q_0 \le -1$.

We begin with a high-level overview of our strategy. Our strategy is different in each of the cases $a > A$, $|a| \le A$, and $a < - A$, where $A>1$ is large real number. 

First, we describe our strategy when $a < - A$. In this case, the position $q$ of the particle will decay to zero rapidly if we just set $u = 0$. This will achieve bounded regret unless $|b| \gg |a|$, in which case the system will decay to zero much faster if we set $u = \pm q$ (depending on the sign of $b$). Therefore, our strategy begins by immediately testing whether $|b|$ is huge. If we find that it is, then we set $u = \pm q$; if not, we set $u = 0$.

Next, we describe our strategy when $|a| \le A$. In this case, we incur a bounded expected cost by simply setting $u=0$. As in the previous case, this achieves bounded regret unless $|b|$ is huge, in which case we should again set $u = \pm q$. Therefore, when $a$ is bounded, our strategy is essentially the same as in the case in which $a$ is large and negative.

Last, suppose that $a>A$. As in the previous two cases, we'd like to set $u = \pm q$ if $|b| \gg |a|$. On the other hand, if $|b|$ is extremely small, then we stand to gain little by controlling the system and we are happy setting $u=0$. For $|b|$ neither huge nor tiny, we can achieve bounded regret by exercising a control whose order of magnitude is $|u| \approx \pm (a/b^2)|q|$ (again, the sign here is determined by the sign of $b$).

We thus proceed as follows. We begin by testing whether $|b| \gg a$. If we detect that it is, then we set $u = \pm q$ (depending on the sign of $b$). If do not detect $|b| \gg a$, then we attempt to learn the order of magnitude of $|b|$. We accomplish this via a series of \emph{Testing Epochs}, in which we exercise increasingly large control over extremely short successive time intervals. If during one of these Testing Epochs we register a significant change in $q$, then we have a good guess for the order of magnitude of $b$ and we control the system by setting $u \approx (a/|b|)|q|$. If, on the other hand, we undergo many Testing Epochs without registering a significant change in $q$, we conclude that $|b|$ is extremely small and we set $u = 0$ for the remainder of the game. We now discuss the Testing Epochs in a bit more detail.

We continue to assume that $a$ is large and positive. Suppose that we have determined that $|b| \gg a$; our goal now is to determine the order of magnitude of $|b|$. We enter a Testing Phase, consisting of a series of Testing Epochs indexed by an integer $\nu\ge 1$. The goal of each Testing Epoch is to determine whether $|b| \approx e^{-\nu}a$. Only the first of these Testing Epochs is guaranteed to occur.

Let $\tau>0$ be a small real number depending on $a$. During Testing Epoch 1, we set $u \approx e (a\tau)^{-1/2}|q|$. If we register a change in $q$ on the order of magnitude of $(a\tau)^{1/2}$ within time $\tau$, then with high probability $|b| \approx e^{-1} a$. Since we have a good guess for $|b|$, we exit the Testing Phase and control the system by setting $u \approx \pm (a/|b|) |q| \approx \pm e^\nu |q|$. If, on the other hand, we fail to register a sufficiently large change in $q$ within time $\tau$, then we assume that $|b|\ll e^{-1}a$ and we enter Testing Epoch 2.

If we enter Testing Epoch 2, then we set $u \approx e^2 (a\tau)^{-1/2}|q|$. As above, if we register a change in $q$ on the order of magnitude of $(a\tau)^{1/2}$ within time $\tau$, then with high probability $|b| \approx e^{-2}a$. We then exit the Testing Phase and control the system by setting $u \approx \pm e^2 |q|$. If, on the other hand, we fail to detect a significant change in $q$ within time $\tau$, then we enter Testing Epoch 3.

In general, if we enter Testing Epoch $\nu$, then we set $u \approx e^\nu (a\tau)^{-1/2}|q|$. If we register a change in $q$ on the order of magnitude of $(a\tau)^{1/2}$ within time $\tau$, then with high probability $|b| \approx e^{-\nu}a$. We then exit the Testing Phase and set $u \approx \pm e^\nu |q|$. If, on the other hand, we fail to detect a significant change in $q$ within time $\tau$, then we enter Testing Epoch $(\nu+1)$.

If we pass through $\nu^*$ Testing Epochs for some large, positive integer $\nu^*$ (depending on $a$) without detecting a significant change in $q$, then we conclude that $|b| \ll e^{-\nu^*}a$. We then achieve bounded regret by setting $u=0$ for the remainder of the game.

This concludes our high-level overview of the strategy $\br$. In order to keep the discussion simple, we've left out some details, e.g., we haven't said how our strategy handles rare events or how we determine the sign of $b$. For details, see Sections \ref{sec: def}--\ref{sec: neg}.

We now outline the contents of the remainder of this paper. In Section \ref{sec: prelim}, we prove some preliminary lemmas about stochastic processes. In Section \ref{sec: strategies}, we establish some basic properties of strategies and of the function $\ecost_\op(b)$. In Sections \ref{sec: conventions}--\ref{sec: neg}, we prove Theorem \ref{thm: main}.

We note here that adaptive control theory is an active field of research, and we refer the interested reader to the literature surveys in our papers \cite{almostoptimal2023, boundeda}.

The author would like to thank the Air Force Office of Scientific Research, and specifically Frederick Leve, for support via AFOSR grant FA9550-19-1-0005. He would also like to thank Maximilian Eggl and Clarence Rowley for helpful discussions, and Charles Fefferman for helpful discussions and for carefully reading an earlier draft of this paper.

\section{Preliminaries}\label{sec: prelim}

Throughout this paper we let $W(t)$ denote a copy of standard Brownian motion and $(\Omega, \cF, \prob)$ denote the corresponding probability space. We write $\E[X]$ and $\var[X]$ to denote the expectation and variance (respectively) of a random variable $X$ with respect to $\prob$. Here, ``standard'' Brownian motion means that $W(0) = 0$ and $\var[W(t)]=t$. For $t\ge 0$, we let $\cF_t$ denote the sigma-algebra determined by the history of $W(s)$ from time $s=0$ until time $s=t$.

Throughout this section we write $C$, $C'$, $c$, $c'$, etc.\ to denote positive absolute constants. We write $C_X$, $C_X'$, $c_X$, etc.\ to denote positive constants depending on a quantity $X$. The value of these constants will change from line to line.

The remainder of this section is devoted to proving two lemmas about the hitting times of certain simple stochastic ODE's. We will make use of the following remark.

\begin{rmk}\label{rmk: xt}
    Suppose $q(t)$ is governed by
    \begin{equation}\label{eq: prelim 1}
    dq(t) = (\alpha q(t) + \beta) dt  + dW(t), \qquad q(0) = Q_0
\end{equation}
for real numbers $\alpha$, $\beta$, and $Q_0 \ne 0$. Define a stochastic process
\begin{equation}\label{eq: prelim 6}
    X(t) = e^{-\alpha t}q(t) - Q_0 = \beta \bigg(\frac{1-e^{-\alpha t}}{\alpha}\bigg) + \int_0^t e^{-\alpha s} dW(s)
\end{equation}
and observe that, for any $t \ge 0$, $X(t)$ is a normal random variable with
\begin{equation}\label{eq: prelim 7}
    \emph{E}[X(t)] = \beta \bigg(\frac{1-e^{-\alpha t}}{\alpha}\bigg), \qquad \emph{Var}[X(t)] = \bigg( \frac{1-e^{-2\alpha t}}{2\alpha}\bigg).
\end{equation}
For the remainder of this paper, we adopt the convention that $(1-e^{-\alpha t})/\alpha = t$ when $\alpha = 0$, so that \eqref{eq: prelim 6} and \eqref{eq: prelim 7} are well-defined for any $\alpha \in\R$.

Observe that the process $t \mapsto (X(t) - \emph{E}[X(t)])$ satisfies the reflection principle, i.e., for any $M>0$ we have
\begin{equation}\label{eq: prelim 4}
\eprob\Big[ \sup_{0 \le s \le t}\{ X(s) - \emph{E}[X(s)]\} \ge M\Big] = 2 \cdot\eprob[ (X(t) - \emph{E}[X(t)]) \ge M].
\end{equation}
The proof of \eqref{eq: prelim 4} is easily adapted from the proof of the reflection principle for Brownian motion; see \cite{almostoptimal2023} for details.
\end{rmk}

We are now ready to state and prove the first of our preliminary lemmas. 

\begin{lem}\label{lem: BKPL}
For any sufficiently large positive integers $m_0$, $n_0$, the following holds. Let $Q_0 \ne 0$, $\alpha >0$, and $\beta$ be real numbers, and let $\tau>0$ be a sufficiently small real number depending on $\alpha$. Suppose $q(t)$ is governed by
\begin{equation}\label{eq: prelim 8}
dq(t) = (\alpha q(t) + \beta Q_0 (\alpha \tau)^{-1/2})dt + dW(t), \qquad q(0) = Q_0.
\end{equation}
Define a stopping time $\tau_*$ to be equal to the first time $t \in (0,\tau)$ for which
\[
|q(t) - Q_0| \ge |Q_0| (\alpha \tau)^{1/2}
\]
if such a time exists and equal to $\tau$ if no such time exists. Then the following hold.
    \begin{enumerate}[label={\emph{(\Alph*)}}]
        \item $\eprob[\tau_* = \tau] \le C \exp(-c Q_0^2 \beta^2/\alpha)$ for any $|\beta| \ge \alpha  e^{n_0}.$
        \item $\eprob[ \tau_* < \tau] \le C \exp(-cQ_0^2 \alpha)$ for any $|\beta| \le \alpha  e^{-m_0}$.
    \end{enumerate}
\end{lem}

\begin{proof}
    Without loss of generality, we assume that $Q_0>0$. Motivated by Remark \ref{rmk: xt}, we define
\[
X(t) = e^{-\alpha t}q(t) - Q_0;
\]
note that
\begin{equation}\label{eq: prelim 9}
\E[X(t)] = \bigg( \frac{\beta Q_0} {(\alpha \tau)^{1/2}}\bigg)\bigg(\frac{1-e^{-\alpha t}}{\alpha}\bigg)
\end{equation}
and
\begin{equation}\label{eq: prelim 10}
    \var[X(t)] = \bigg(\frac{1-e^{-2\alpha t}}{2\alpha }\bigg).
\end{equation}
Provided $\tau$ is sufficiently small depending on $\alpha$, we then have for any $t \in [0,\tau]$ that
\begin{align}
     &|\E[X(t)]| \approx \bigg( \frac{|\beta| Q_0  t}{(\alpha \tau)^{1/2}}\bigg),\label{eq: prelim 13}\\
    & \var[X(t)] \approx  t.\label{eq: prelim 8.5}
\end{align}

Let $n_0 \ge 1$ be a sufficiently large integer and assume until further notice that
\begin{equation}\label{eq: prelim 14}
    \beta \ge \alpha  e^{n_0}.
\end{equation}

Suppose that $q(\tau) < Q_0(1+(\alpha \tau)^{1/2})$. This implies that $X(\tau) < Q_0 (\alpha \tau)^{1/2}$; combining this with \eqref{eq: prelim 13} and \eqref{eq: prelim 14} (and observing that $\E[X(\tau)]>0$ when $\beta>0$), we get
\begin{equation*}
    (X(\tau) - \E[X(\tau)]) < \frac{Q_0 \tau^{1/2} \beta}{\alpha^{1/2}}(e^{-n_0} - c).
\end{equation*}
Taking $n_0$ sufficiently large ensures that
\begin{equation*}
    (X(\tau) - \E[X(\tau)]) < - c \bigg( \frac{Q_0 \tau^{1/2} \beta}{\alpha^{1/2}}\bigg).
\end{equation*}
Note that the event $(\tau_* = \tau)$ implies that $q(t) < Q_0(1+(\alpha\tau)^{1/2})$ for all $t \in (0,\tau)$. We have therefore shown that
\begin{equation}\label{eq: prelim 15}
    \prob[ \tau_* = \tau ] \le \prob\big[(X(\tau) - \E[X(\tau)]) < - c Q_0 \tau^{1/2} \beta / \alpha^{1/2}\big].
\end{equation}
Since $X(\tau)$ is a normal random variable, we use \eqref{eq: prelim 8.5} and \eqref{eq: prelim 15} to deduce that
\begin{equation}\label{eq: prelim 16}
    \prob[\tau_* = \tau] \le C \exp(-c Q_0^2 \beta^2 /\alpha)\;\text{for any}\; \beta \ge \alpha  e^{n_0}.
\end{equation}

We now assume that
\begin{equation}\label{eq: prelim 17}
    \beta \le - \alpha  e^{n_0}.
\end{equation}
If $q(\tau) > Q_0(1- (\alpha\tau)^{1/2})$, then
\begin{align*}
X(\tau) >& (1-\alpha\tau)Q_0(1-(\alpha\tau)^{1/2}) - Q_0\\
>& - c(\alpha\tau)^{1/2}Q_0
\end{align*}
provided $\tau$ is sufficiently small depending on $\alpha$. Combining this with \eqref{eq: prelim 17} and \eqref{eq: prelim 13} (and observing that $\E[X(t)] < 0$ when $\beta<0$), we get
\[
(X(\tau) - \E[X(\tau)]) > \frac{Q_0\tau^{1/2}|\beta|}{\alpha^{1/2}}( c- c'e^{-n_0}).
\]
Taking $n_0$ sufficiently large then ensures that
\[
(X(\tau) - \E[X(\tau)] ) > c \frac{Q_0 \tau^{1/2}|\beta| }{\alpha^{1/2}}.
\]
Proceeding as in the proof of \eqref{eq: prelim 16}, i.e., observing that the event $(\tau_* = \tau)$ implies $q(\tau) > Q_0(1- (\alpha\tau)^{1/2})$ and using \eqref{eq: prelim 8.5} along with the fact that $X(\tau)$ is a normal random variable, we deduce that
\begin{equation}\label{eq: prelim 18}
    \prob[\tau_* = \tau] \le C \exp(-c Q_0^2 \beta^2/\alpha) \;\text{for any}\; \beta \le - \alpha  e^{n_0}.
\end{equation}
Combining \eqref{eq: prelim 16} and \eqref{eq: prelim 18} proves Part (A) of the Lemma.

Until further notice, assume that $\beta \ge 0$. Observe that if $q(t) \le Q_0 (1- (\alpha\tau)^{1/2})$ for some $t \in [0,\tau]$, then $X(t) \le -  (\alpha\tau)^{1/2}Q_0$. Since $\E[X(t)] \ge 0$ (see \eqref{eq: prelim 9}), we therefore have
\begin{multline*}
\prob[\exists t \in [0,\tau] : q(t) \le Q_0 (1- (\alpha\tau)^{1/2})]\\ \le \prob[ \exists t \in [0,\tau] : (X(t) - \E[X(t)]) \le -  (\alpha\tau)^{1/2} Q_0].
\end{multline*}
By the reflection principle for $(X(t) - \E[X(t)])$, \eqref{eq: prelim 8.5}, and the fact that $X(\tau)$ is a normal random variable, we have
\begin{equation}\label{eq: prelim 11}
\prob[\exists t \in [0, \tau] : q(t) \le Q_0 (1-(\alpha\tau)^{1/2})] \le C\exp(-c Q_0^2 \alpha)\;\text{for any}\; \beta \ge 0.
\end{equation}

Now assume that $\beta < 0$. If $q(t) \ge Q_0(1+(\alpha\tau)^{1/2})$, for some $t \in [0, \tau]$, then
\begin{align*}
X(t) \ge& (1-\alpha t)Q_0(1+(\alpha\tau)^{1/2}) - Q_0\\
\ge& c (\alpha \tau)^{1/2}Q_0
\end{align*}
provided $\tau$ is sufficiently small depending on $\alpha$. By \eqref{eq: prelim 9}, $\E[X(t)] \le 0$ when $ \beta<0$. We have therefore shown that
\begin{multline*}
\prob[ \exists t \in [0, \tau] : q(t) \ge Q_0 (1+(\alpha\tau)^{1/2}) ]\\ \le \prob[ \exists t \in [0, \tau] : (X(t) - \E[X(t)]) \ge  c (\alpha\tau)^{1/2}Q_0 ] .
\end{multline*}
As above, we use the reflection principle for $(X(t) - \E[X(t)])$, along with \eqref{eq: prelim 8.5} and the fact that $X(\tau)$ is a normal random variable, to deduce that
\begin{equation}\label{eq: prelim 12}
    \prob[\exists t \in [0, \tau] : q(t) \ge Q_0 (1+(\alpha\tau)^{1/2})] \le C \exp(-c Q_0^2 \alpha)\;\text{for any} \; \beta < 0.
\end{equation}

Let $m_0\ge 1$ be a sufficiently large integer and assume until further notice that
\begin{equation}\label{eq: prelim 20}
    \alpha e^{-m_0} \ge \beta \ge 0.
\end{equation}
Observe that if $q(t)> Q_0(1+(\alpha\tau)^{1/2})$ for any $t \in [0,\tau]$, then $X(t) > c(\alpha \tau)^{1/2}Q_0$ provided $\tau$ is sufficiently small depending on $\alpha$. Combining this with \eqref{eq: prelim 13} and \eqref{eq: prelim 20} gives
\[
(X(t) - \E[X(t)]) >  (\alpha \tau)^{1/2} Q_0 ( c -  Ce^{-m_0});
\]
taking $m_0$ sufficiently large gives
\[
(X(t) - \E[X(t)]) > c(\alpha\tau)^{1/2}Q_0.
\]
We have therefore shown that
\begin{multline*}
\prob[\exists t \in [0,\tau] : q(t) > Q_0(1+(\alpha\tau)^{1/2})]\\ \le \prob[ \exists t \in [0, \tau] : (X(t) - \E[X(t)]) > c (\alpha\tau)^{1/2}Q_0].
\end{multline*}
Using \eqref{eq: prelim 4} (the reflection principle for $X(t)$), along with \eqref{eq: prelim 8.5} and the fact that $X(\tau)$ is a normal random variable, we have
\begin{equation}\label{eq: prelim 17.5}
\begin{split}
    \prob[\exists t \in [0,\tau] : q(t) > Q_0(&1+(\alpha\tau)^{1/2})]\\ &\le C \exp(-c Q_0^2 \alpha)\; \text{for any}\; 0 \le \beta \le \alpha e^{-m_0}.
    \end{split}
\end{equation}
Combining this with \eqref{eq: prelim 11} gives
\begin{equation}\label{eq: prelim 21}
\prob[\tau_* < \tau] \le C \exp(-c Q_0^2 \alpha) \;\text{for any}\; 0 \le \beta \le \alpha  e^{-m_0}.
\end{equation}

Now assume that
\begin{equation}\label{eq: prelim 22}
- \alpha  e^{-m_0} \le \beta < 0.
\end{equation}
Observe that if
\[
q(t) < Q_0(1- (\alpha\tau)^{1/2}) \;\text{for any}\; t \in [0,\tau]
\], then $X(t) < -Q_0 (\alpha\tau)^{1/2}$. Combining this with \eqref{eq: prelim 13} and \eqref{eq: prelim 22}, we get
\[
(X(t) - \E[X(t)]) < - Q_0 (\alpha\tau)^{1/2}(  c- Ce^{-m_0}).
\]
Taking $m_0$ sufficiently large gives
\[
(X(t) - \E[X(t)]) < - c Q_0(\alpha\tau)^{1/2}.
\]
Proceeding in the proof of \eqref{eq: prelim 17.5}, we deduce that
\begin{multline*}
\prob[\exists t \in (0,\tau) : q(t) < Q_0(1-(\alpha\tau)^{1/2})]\\ \le C \exp(-c Q_0^2 \alpha) \;\text{for any}\; - \alpha  e^{-m_0} \le \beta < 0.
\end{multline*}
Combining this with \eqref{eq: prelim 12} gives
\begin{equation}\label{eq: prelim 23}
    \prob[\tau_* < \tau] \le C \exp(-c Q_0^2 \alpha) \;\text{for any}\; -\alpha  e^{-m_0}\le \beta < 0.
\end{equation}
Together, \eqref{eq: prelim 21} and \eqref{eq: prelim 23} prove Part (B) of the lemma.
\end{proof}

\begin{lem}\label{lem: nhl}
    Let $Q_0 \ne 0$ and $\alpha$ be real numbers. Suppose that $q(t)$ is governed by
    \[
    dq(t) = \alpha q(t) dt + dW_t, \qquad q(0) = Q_0.
    \]
    For any $\eta \in (0, (\log(2))^{-1})$, the following hold.
    \begin{enumerate}[label=\emph{(\Alph*)},ref=\Alph*]
    \item If $\alpha>0$, then:
    \begin{enumerate}[label=\emph{(\roman*)},ref={A.\roman*}]
        \item\label{lem: nhl ai} $\eprob[q(t) < 2Q_0 \;\emph{for all}\; t \in [0, (\eta \alpha)^{-1}]] \le C \exp(-c_\eta Q_0^2 \alpha)$.
        \item\label{lem: nhl aii} $\eprob[\exists t \ge 0 : q(t) \le Q_0 / 2] \le C \exp(-c Q_0^2  \alpha)$.
    \end{enumerate}
    \item If $\alpha< 0$, then:
    \begin{enumerate}[label=\emph{(\roman*)},ref={B.\roman*}]
    \item\label{lem: nhl bi} $\eprob[q(t) > Q_0/2 \;\emph{for all}\; t \in [0, (\eta |\alpha|)^{-1}]] \le C \exp(-c_\eta Q_0^2|\alpha|)$.
    \item\label{lem: nhl biii} For any $\hat{T}>0$, we have
    \[
    \eprob[\exists t \in [0,\hat{T}] : |q(t)| \ge 2Q_0] \le C_{\hat{T}} (1+Q_0^{-2}) \exp(-c Q_0^2 |\alpha|).
    \]
    \end{enumerate}
    \item\label{lem: nhl c} If $\alpha \ne 0$, then
    \[
    \eprob[ \exists t \in [0, \eta |\alpha|^{-1}] : q(t) \ge 2Q_0\;\emph{OR}\; q(t) \le Q_0/2] \le C \exp(-c_\eta Q_0^2 |\alpha|).
    \]
    \end{enumerate}
\end{lem}

\begin{proof}
    We first remark that Part (B.ii) is an easy consequence of Lemma 4.3 in \cite{almostoptimal2023}.

    Motivated by Remark \ref{rmk: xt}, we define
    \[
    X(t) = e^{-\alpha t} q(t) - Q_0,
    \]
    and we recall that $X(t)$ is a normal random variable with mean 0 and
    \begin{equation}\label{eq: nhl 1}
    \var[ X(t)] = \frac{1-e^{-2\alpha t}}{2\alpha}.
    \end{equation}

    Assume until further notice that $\alpha > 0$. If $q((\eta \alpha)^{-1}) < 2 Q_0$, then
\[
X((\eta \alpha)^{-1}) < e^{-1/\eta}2Q_0 - Q_0 < - c_\eta Q_0.
\]
Therefore
\[
\prob[q(t) < 2Q_0 \;\text{for all}\; t \in [0,(\eta \alpha)^{-1}]] \le \prob[ X((\eta \alpha )^{-1}) < - c_\eta Q_0].
\]
By \eqref{eq: nhl 1}, 
\begin{equation}\label{eq: nhl 1.5}
    \var[X((\eta \alpha)^{-1})] \approx \alpha^{-1}.
\end{equation}
Since $X((\eta \alpha)^{-1})$ is a normal random variable with mean 0, we conclude that
\begin{equation}\label{eq: nhl 2}
\prob[q(t) < 2 Q_0 \;\text{for all}\; t \in  [0, (\eta \alpha)^{-1}]] \le C \exp(-c_\eta Q_0^2 \alpha).
\end{equation}
This proves (A.i)

To prove (A.ii), it suffices (e.g., by Fatou's Lemma) to show that for any sufficiently large $\hat{T}>0$ we have
\begin{equation}\label{eq: nhl 3}
\prob[ \exists t \in [0,\hat{T}] : q(t) \le Q_0/2] \le C \exp(-c Q_0^2 \alpha).
\end{equation}
If $q(t) \le Q_0/2$ for some $t \in [0, \hat{T}]$, then $X(t) \le - Q_0/2$. Combining this with \eqref{eq: prelim 4} (i.e., the reflection principle for $X(t)$), we see that
\[
\prob[\exists t \in [0, \hat{T}] : q(t) \le Q_0 / 2] \le 2\cdot \prob[ X(\hat{T}) < - Q_0 / 2].
\]
To prove \eqref{eq: nhl 3}, we note that $\var[X(\hat{T})] \approx \alpha$ and that $X(\hat{T})$ is a normal random variable with mean 0.

Now observe that if $q(t) \ge 2Q_0$ for some $t \in [0, \eta \alpha^{-1}]$, then
\[
X(t) \ge e^{-\alpha t} 2 Q_0 - Q_0 \ge c_\eta Q_0.
\]
By \eqref{eq: prelim 4} (i.e., the reflection principle for $X(t)$), we have
\[
\prob[\exists t \in [0, \eta \alpha^{-1}] : q(t) \ge 2Q_0] \le 2\cdot \prob[ X(\eta \alpha^{-1}) \ge c_\eta Q_0].
\]
Observe that $\var[X(\eta \alpha^{-1})] \approx_\eta \alpha^{-1}$; combining this with the fact that $X(\eta \alpha^{-1})$ is a normal random variable with mean 0 gives
\begin{equation}\label{eq: nhl 4}
\prob[\exists t \in [0, \eta \alpha^{-1}] : q(t) \ge 2Q_0] \le C \exp(-c_\eta Q_0^2 \alpha).
\end{equation}
We combine \eqref{eq: nhl 3} and \eqref{eq: nhl 4} to get
\begin{equation}\label{eq: nhl 5}
\begin{split}
    \prob[\exists t \in [0, \eta \alpha^{-1}] : (q(t) \ge 2Q_0)\;\text{OR}\; &(q(t) \le Q_0 / 2)]\\ &\le C \exp(-c_\eta Q_0^2 \alpha)\;\text{for any}\; \alpha > 0.
\end{split}
\end{equation}
This proves (C) when $\alpha>0$.

We omit the proofs of (B.i), and (C) when $\alpha<0$; they are essentially the same as the proofs of \eqref{eq: nhl 2} and \eqref{eq: nhl 5} (respectively).
\end{proof}

\section{Agnostic Control Strategies}\label{sec: strategies}
We begin this section by establishing some conventions and notation:
\begin{itemize}
    \item We introduce a \emph{starting position} $Q_0 \ne 0$, a \emph{time horizon} $T>0$, and a \emph{growth parameter} $a \in \R$.
    \item Constants $c,C,C',$ etc.\ denote absolute constants.
    \item \sloppypar Recall that $W(t)$ denotes a copy of standard Brownian motion, $(\Omega, \cF, \prob)$ denotes the corresponding probability space, and, for any $t \ge 0$, we write $\cF_t$ to denote the sigma-algebra determined by the history $(W(s))_{s \in [0,t]}$.
\end{itemize}

A \emph{strategy} (for starting position $Q_0$, time horizon $T$, and growth parameter $a$) is a collection of random variables $q(t,b)$ and $u(t,b)$ defined for all $t \in [0,T]$ and $b \in \R$ satisfying the following properties:
\begin{enumerate}[label=(S.\arabic*)]
\item For every $b \in \R$, $q(t,b)$ is a continuous function of $t$ with probability 1 and $u(t,b)$ is an $L^2$ function of $t$ with probability 1.
\item\label{s4} For every $b \in \R$ and $t \in [0,T]$ the maps $(s,\omega) \mapsto q(s,b,\omega)$ and $(s,\omega) \mapsto u(s,b,\omega)$, defined on $[0,t]\times \Omega$, are measurable as functions on $[0,t]\times (\Omega, \cF_t, \prob)$. Intuitively, this means that $q,u$ are determined by the past.
 \item For every $b \in \R$,
    \[
    \E\bigg[ \int_0^{T} ( q(t,b))^2 + (u(t,b))^2\ dt \bigg] < \infty.
    \]
    \item For almost all $\omega \in \Omega$, we have that for all $b,b'\in \R$, and for all $t \in [0,T]$, if $q(s,b,\omega) = q(s,b',\omega)$ for all $s \in [0,t]$, then $u(s,b,\omega) = u(s,b',\omega)$ for almost all $s \in [0,t]$. This tells us that $u$ does not depend on the unknown $b$.
    \item\label{s2} For every $b \in \R$ and $t \in [0,T]$, we have
    \[
    q(t,b) = Q_0 + W(t) + \int_0^t [a q(s,b) + bu(s,b)]\ d s
    \]
    with probability 1.
\end{enumerate}
For a given $b \in \R$, we refer to $q(t,b)$ and $u(t,b)$ as, respectively, the \emph{particle trajectory} and the \emph{control variable} at time $t$. 

Let $\sigma$ denote an arbitrary strategy. We then write $q^\sigma(t,b)$, $u^\sigma(t,b)$ for the particle trajectories and control variables associated with $\sigma$. Often, the parameter $b$ is implicit and we just write $q^{\sigma}(t)$, $u^\sigma(t)$.

We define the \emph{cost} of a strategy $\sigma$ to be the random variable
\[
\cost(\sigma, b) = \int_0^{T} \big(( q^\sigma(t,b))^2 + (u^\sigma(t,b))^2\big)\ dt.
\]
We then define the \emph{expected cost} of a strategy $\sigma$ by
\[
\ecost(\sigma,b) = \E[\cost(\sigma,b)].
\]
Given any smooth function $v:[0,T]\rightarrow \R$, we can define a strategy $\tilde{\sigma}$ by setting
\[
u^{\tilde{\sigma}}(t,b) = - v(t) q^{\tilde{\sigma}}(t,b)
\]
(note that $q^{\tilde{\sigma}}$ is then determined by Property (S.5)). We refer to such a strategy $\tilde{\sigma}$ as a \emph{simple feedback strategy} with \emph{gain function} $v$.

We now establish some upper bounds on the expected cost of simple feedback strategies.

Let $\tilde{\sigma}$ denote the simple feedback strategy with gain function $v$. One can show (as in \cite{almostoptimal2023}) that for any $b \in \R$ we have
\begin{equation*}
\ecost(\tilde{\sigma},b) = \varphi(0,b) \cdot Q_0^2 + \int_0^{T} \varphi(t,b) \ dt,
\end{equation*}
where
\[
\varphi(t,b) = \int_t^{T} (1+v^2(\tau))\cdot \exp\bigg( 2 \int_t^\tau \big(\alpha - b v(s)\big) \ ds \bigg)\ d\tau.
\]
This implies the following remark.

\begin{rmk}\label{rmk: cg}
Let $\ecg(\alpha)$ denote the simple feedback strategy with constant gain function $v \equiv \alpha \in \R$ (so that $u^{\ecg(\alpha)}(t) = - \alpha\cdot( q^{\ecg(\alpha)}(t))$). Then the following hold.
\begin{enumerate}[label={\emph{(\Alph*)}}, ref={\Alph*}]
    \item\label{rmk: cg a} If $a<b\alpha$, then
    \[
    \ecost(\ecg(\alpha),b) \le \frac{(Q_0^2+T) (1+\alpha^2)}{|a - b \alpha|}.
    \] 
    \item\label{rmk: cg b} Let $A= \max\{|\alpha|, |a -b \alpha|\}$. Then
    \[
    \ecost(\ecg(\alpha),b) \le C_A(Q_0^2 + T).
    \]
    \item\label{rmk: cg c} If $a > 0$ and $\alpha = 0$, then
    \[
    \ecost(\ecg(0),b) \le (Q_0^2 + T) \frac{e^{2aT}}{2a}\;\text{for any}\; b \in \R.
    \]
\end{enumerate}
\end{rmk}

For any $\beta \in \R$, we define a simple feedback strategy $\sigma_\op(\beta)$ via the control variable
\[
u^{\sigma_\op(\beta)}(t) = - \beta \cdot \kappa(T-t,\beta) \cdot q^{\sigma_\op(\beta)}(t),
\]
where $\kappa:[0,T]\times \R \rightarrow [0,\infty)$ is defined by
\begin{equation}\label{eq: cc 2}
\kappa(t,\beta) = \kappa(t,\beta;a) = \begin{cases}
t & \text{if } a=\beta=0,\\
    \frac{\tanh(t\sqrt{a^2+\beta^2})}{\sqrt{a^2+\beta^2} - a \tanh(t\sqrt{a^2+\beta^2})} & \text{else}.
\end{cases}
\end{equation}

It is well-known (see, e.g., \cite{astrom}) that for any $b \in \R$ the strategy $\sigma_\op(b)$ minimizes the quantity $\ecost(\sigma,b)$ among simple feedback strategies. (In fact, by a straightforward modification of an argument in \cite{almostoptimal2023}, one can show that $\sigma_\op(b)$ minimizes $\ecost(\sigma,b)$ among \emph{all} strategies.)  For $b\in\R$ we define
\[
\ecost_\op(b) = \ecost(\sigma_\op(b),b).
\]
Of course, both the strategies $\sigma_\op(\beta)$ and the quantity $\ecost_\op(b)$ depend on the parameters $a$, $T$, $Q_0$.

It is well known (again, see \cite{astrom}) that $\ecost_\op(b)$ is of the form
\begin{equation}\label{eq: cc 3}
\ecost_\op(b) = \kappa(T,b)\cdot Q_0^2 + \int_0^T \kappa(t,b)\ dt,
\end{equation}
where $\kappa(t, b)$ is defined in \eqref{eq: cc 2}. We use \eqref{eq: cc 3} to prove the following lemma.
\begin{lem}\label{lem: j0}
We have the following lower bounds on $\ecost_\eop(b)$.
    \begin{enumerate}[label={\emph{(\Alph*)}}]
        \item If $a \ge 1$, then
    \[
    \ecost_\eop(b) \ge c_{T}\cdot Q_0^2 \cdot \begin{cases}
         \frac{e^{2aT}}{a} &\text{for any}\; |b| \le  a e^{-aT},\\
        \frac{a}{b^2} &\text{for any}\;  a e^{-aT} \le |b| \le a,\\
        \frac{1}{|b|} &\text{for any}\; |b| \ge a.
    \end{cases}
    \]
    \item If $a \le 1$, then
    \[
    \ecost_\eop(b) \ge \frac{ c_{T}\cdot Q_0^2}{1+|a|+|b|}\;\text{for any}\; b \in \R.
    \]
    \end{enumerate}
\end{lem}
\begin{proof}
Observe that
    \[
    \kappa(t,b) \ge 0 \;\text{for any}\; t \in [0,T],
    \]
    and therefore
    \begin{equation}\label{eq: cc 3.5}
        \ecost_\op(b) \ge Q_0^2 \cdot \kappa(T,b).
    \end{equation}
It is straightforward to verify that
\[
\kappa(T,b) > \frac{c_T}{1+ |a|+|b|}\;\text{for any}\; a,b \in \R;
\]
combining this with \eqref{eq: cc 3.5} implies Part (B) of the Lemma as well as the case $|b| \ge a$ in Part (A).

    We now remark that
\begin{equation}\label{eq: cc 6}
\begin{split}
        &\tanh(x) \ge (1 - C e^{-2x}) \;\text{for any}\; x \ge 0.
\end{split}
\end{equation}
Assume that $a \ge 1$. Then clearly
\[
\kappa(T,b) \ge \frac{c_T}{\sqrt{a^2+b^2} - a \tanh(T\sqrt{a^2+b^2})}\;\text{for any}\; b \in \R.
\]
Combining this with \eqref{eq: cc 6} gives
\[
\kappa(T,b)  \ge \frac{c_T}{ (b^2/a) + aC e^{-2T\sqrt{a^2+b^2}}}.
\]
From this inequality and \eqref{eq: cc 3.5}, we easily deduce the remaining two cases of Part (A).
\end{proof}

\section{Conventions for the remainder of the paper}\label{sec: conventions}

The remainder of this paper is devoted to proving Theorem \ref{thm: main}.

Throughout Sections \ref{sec: def}--\ref{sec: neg}, the following hold:
\begin{itemize}
    \item We fix a starting position $|q_0|\ge 1$ and a time horizon $T>0$.
    \item For any $a,b \in \R$, we write $\ecost_\op(b;a)$ to denote the optimal expected cost for known $b$ with time horizon $T$, starting position $q_0$, and growth parameter $a$. Similarly, we write $\ecost(\sigma,b;a)$ to denote the expected cost of a strategy $\sigma$ for time horizon $T$, starting position $q_0$, and growth parameter $a$.
    \item Often, the value of the growth parameter $a$ will be clear from context and we will just write $\ecost_\op(b)$, $\ecost(\sigma, b)$.
    \item We let $A\ge 1$ be a large real number depending only on $T$.
    \item Constants $c,C,C',$ etc.\ are allowed to depend on $q_0$ and $T$ (and therefore on $A$).
\end{itemize}

For any $a \in \R$, we will define in Sections \ref{sec: def}--\ref{sec: neg} a strategy denoted $\br$ for time horizon $T$, starting position $q_0$, and growth parameter $a$. We will then show that there exists a constant $C>1$ so that
\[
    \ecost(\br,b;a) \le C \cdot\ecost_\op(b;a)\;\text{for any}\; a,b \in \R.
\]
Once we do this, we've proved Theorem \ref{thm: main}.

\section{The strategy for large, positive $a$}\label{sec: def}

We let $a$ denote the growth parameter, and we assume throughout this section that
\[
a \ge A.
\]
Our goal in this section is to prove Theorem \ref{thm: main} for such $a$.

We let $0 < \tau \ll 1$ denote a small real number depending on $a$.

We let $K \ge 1$ denote a large absolute constant, to be chosen later. Of course, constants $c, C, C',$ etc.\ may depend on $K$.

In a moment we will define the strategy $\br$ (referenced in Theorem \ref{thm: main}) in the case $a \ge A$. Throughout this section we write $q,u$ for the particle trajectories and control variables $q^{\br},u^{\br}$.

We define a quantity
\begin{equation}\label{eq: nu star def}
\nu^* = \lfloor a T \rfloor.
\end{equation}
The strategy $\br$ is divided into $(\nu^*+2)$ Epochs; we enumerate them starting from 0. Only the first of the Epochs (Epoch 0) is guaranteed to occur. Each of the Epochs $0,\dots,\nu^*$ (all but the final Epoch) is divided into two or more Subepochs. When one of these Epochs occurs, only the first Subepoch (the \emph{Testing} Subepoch) is guaranteed to occur.

We now define the strategy $\br$.

\noindent \underline{Epoch 0}: Our strategy begins, at time 0, with Subepoch $0.\I$.

\underline{Subepoch $0.\I$ (Testing)}: During Subepoch $0.\I$ we set $u=-q$. Subepoch $0.\I$ ends at time $\tau_0$, where $\tau_0$ is equal to the first time $t \in (0,(10a)^{-1})$ for which $q(t) = 2q_0$ or $q(t) = \frac{1}{2} q_0$ if such a time exists and is equal to $(10a)^{-1}$ if no such time exists.

In the event that $\tau_0 < (10a)^{-1}$, then at time $\tau_0$ we enter Subepoch $0.\ii$. In the event that $\tau_0 = (10a)^{-1}$, then at time $\tau_0$ we enter Epoch 1.

We stipulate that $A$ is large enough to ensure that $(10a)^{-1} < T$, so that $\tau_0 < T$ with probability 1.

\underline{Subepoch $0.\ii$ (Control)}: Note that $\tau_0$ is equal to the time at which we enter Subepoch $0.\ii$ if such a time exists and equal to $T$ if no such time exists.

Suppose that we enter Subepoch $0.\ii$, i.e., suppose that $\tau_0 < (10a)^{-1}$; we then have either $q(\tau_0) = \frac{1}{2} q_0$ or $q(\tau_0) = 2q_0$ with probability 1.

If $q(\tau_0)= \frac{1}{2}q_0$, then during Subepoch $0.\ii$ we set $u(t) = -K\cdot q(t)$.

If $q(\tau_0)= 2q_0$, then during Subepoch $0.\ii$ we set $u(t) = K\cdot q(t)$.

Subepoch $0.\ii$ then lasts until time $\titau_0$, where $\titau_0$ is equal to the first time $t \in (\tau_0, T)$ for which $|q(t)| \ge 2 q(\tau_0)$ if such a time exists and is equal to $T$ if no such time exists. If $\titau_0 =T$, then the game ends along with Subepoch $0.\ii$ at time $\titau_0$. If $\titau_0<T$, then at time $\titau_0$ we enter Epoch 1.

\noindent \underline{Epoch $\nu$ (for $1\le \nu \le \nu^*$)}:

Define a stopping time $t_\nu$ to be equal to the time at which we enter Epoch $\nu$ if such a time exists and equal to $T$ if no such time exists.

Epoch $\nu$ is divided into three Subepochs, only the first of which (Subepoch $\nu.\I$) is guaranteed to occur.

\underline{Subepoch $\nu.\I$ (Testing):} Assume that $t_\nu <T$; we then enter Subepoch $\nu.\I$ at time $t_\nu$. During Subepoch $\nu.\I$ we set
\[
u = \frac{e^\nu q(t_\nu)}{(a\tau)^{1/2}},
\]
and Subepoch $\nu.\I$ lasts until time $\tau_\nu$, where $\tau_\nu$ is equal to the first time $t$ in the interval $(t_\nu, \min\{T, t_\nu + \tau\})$ for which
\[
|q(t) - q(t_\nu)| >  |q(t_\nu)| (a\tau)^{1/2}
\]
if such a time exists and equal to $\min\{ T, t_\nu + \tau\}$ if no such time exists.

If $\tau_\nu < \min\{T, t_\nu + \tau\}$, then at time $\tau_\nu$ we enter Subepoch $\nu.\ii$.

If $\tau_\nu = (t_\nu + \tau)$, then at time $\tau_\nu$ we enter Epoch ($\nu$ + 1).

If $\tau_\nu = T$, then at time $\tau_\nu$ the game ends.

\underline{Subepoch $\nu.\ii$ (Control I)}:
Define a stopping time $\tit_\nu$ to be equal to the time at which we enter Subepoch $\nu.\ii$ if such a time exists and equal to $T$ if no such time exists. Note that we enter Subepoch $\nu.\ii$ only if we enter Subepoch $\nu.\I$.

Suppose that we enter Subepoch $\nu.\ii$, i.e., suppose that $\tit_\nu<T$. Then during Subepoch $\nu.\ii$ we set $u = 1000 e^\nu q$, and Subepoch $\nu.\ii$ lasts until time $\titau_\nu$, where $\titau_\nu$ is equal to the first time $t \in (\tit_\nu, T)$ for which $|q(t)| \ge 2 |q(\tit_\nu)|$ if such a time exists and equal to $T$ if no such time exists.

If $\titau_\nu < T$, then at time $\titau_\nu$ we enter Subepoch $\nu.\iii$.

If $\titau_\nu = T$, then at time $\titau_\nu$ the game ends (along with Subepoch $\nu.\ii$).

\underline{Subepoch $\nu$.iii (Control II)}: Define a stopping time $\htt_\nu$ to be equal to the time at which we enter Subepoch $\nu.\iii$ if such a time exists and equal to $T$ if no such time exists. Note that we enter Subepoch $\nu.\iii$ only if we enter Subepoch $\nu.\ii$.

Suppose that we enter Subepoch $\nu$.iii. Then during Subepoch $\nu$.iii, we set $u = - 1000 e^\nu q$ and Subepoch $\nu$.iii lasts until time $\httau_\nu$, where $\httau_\nu$ is equal to the first time $t \in (\htt_\nu, T)$ for which $|q(t)| \ge 2 |q(\htt_\nu)|$ if such a time exists and equal to $T$ if no such time exists.

If $\httau_\nu < T$, then at time $\httau_\nu$ we enter Epoch $(\nu+1)$.

If $\httau_\nu = T$, then at time $\httau_\nu$ the game ends (along with Subepoch $\nu.\iii$).

\noindent \underline{Epoch $(\nu^*+1)$ (Apathy)}:  Define a stopping time $t_{\nu^*+1}$ to be equal to the time at which we enter Epoch $(\nu^*+1)$ if such a time exists and equal to $T$ if no such time exists. If we enter Epoch $(\nu^*+1)$, then we set $u = 0$ until the end of the game at time $T$.

This completes the definition of our strategy $\br$. It is straightforward to verify that the strategy $\br$ satisfies the definition of a strategy given in Section \ref{sec: strategies}; we will not give the details here. (For more detail see the discussion of branching strategies in \cite{almostoptimal2023}.)

We say that a given Epoch or Subepoch occurs if there exists some time at which we enter the given Epoch or Subepoch. 

For $0 \le \nu \le (\nu^*+1)$ we define, for a given value of $b$, the random variable $\cost_\nu(b)$ to be equal to the cost incurred during Epoch $\nu$ if Epoch $\nu$ occurs and equal to 0 if Epoch $\nu$ does not occur. Clearly
\[
\ecost(\br,b) = \sum_{\mu=0}^{\nu^*+1} \E[\cost_\mu(b)].
\]

For $0 \le \nu \le \nu^*$, we define the random variable $\cost_{\nu.j}(b)$ to be equal to the cost incurred during Subepoch $\nu.j$ (with $j\in \{\I,\ii,\iii\}$ if $\nu \ge 1$ and $j \in \{\I,\ii\}$ if $\nu = 0$) if Subepoch $\nu.j$ occurs and equal to zero if Subepoch $\nu.j$ does not occur.

We define events
\begin{equation*}
\begin{aligned}
\cE_\nu &= \{\text{Epoch}\; \nu \;\text{occurs}\} &&\text{for}\; 0 \le \nu \le (\nu^*+1),\\
\cE_{\nu.\text{ii}} &= \{\text{Subepoch}\;\nu.\text{ii occurs}\}\qquad &&\text{for} \; 0 \le \nu \le \nu^* ,\\
\cE_{\nu.\text{iii}} &= \{\text{Subepoch}\; \nu.\text{iii occurs}\}\qquad &&\text{for}\; 1 \le \nu \le \nu^*.
\end{aligned}
\end{equation*}
Note that
\begin{equation*}
\begin{aligned}
    & \cE_{\nu.\iii}\subset\cE_{\nu.\ii}\subset\cE_\nu& &\text{for}\; 1\le \nu \le \nu^*,\\
    &\cE_\nu \subset \bigcap_{\mu=0}^{\nu-1}\cE_\mu& &\text{for}\; 1\le \nu \le (\nu^*+1).
\end{aligned}
\end{equation*}
For any $1 \le \nu \le (\nu^*+1)$ and $0\le m \le \nu$ we define the event
\[
\cE_{\nu,m} = \{ \cE_\nu \text{ AND } (\cE_{\mu.\ii}\;\text{for exactly }m\text{ integers } \mu \text{ with } 0 \le \mu < \nu)\}.
\]
Note that by taking $\tau$ sufficiently small depending on $a$, we can ensure that
\begin{align*}
&(1+(a\tau)^{1/2})^{\nu^*} < \frac{11}{10},\\
&(1-(a\tau)^{1/2})^{\nu^*} > \frac{9}{10}
\end{align*}
(recall that the quantity $\nu^*$ is determined by $a$). This implies the following: Suppose that the event $\cE_{\nu,m}$ occurs and that we enter Epoch $\nu$ at time $t_\nu$. Then with probability 1 we have
\begin{equation}\label{eq: q bound}
c\cdot 4^{m} \le |q(t_\nu)| \le C\cdot 4^m.
\end{equation}

For any stopping time $\tau^*$ we define
\[
\cF_{\tau^*} = \{ A \in \cF : A \cap \{\tau^* \le s \} \in \cF_s \;\text{for all}\; s \in [0,T]\}.
\]
Note that $\cF_{\tau^*}$ is a sub $\sigma$-algebra of $\cF$ and that the random variables $\tau^*$ and $(q(s))_{0 \le s \le \tau^*}$ are $\cF_{\tau^*}$-measurable.

\subsection{Performance of the strategy during Epoch 0}\label{sec: 0}

Recall that we set $u = - q$ during Subepoch $0.\I$ and that Subepoch $0.\I$ ends at time $\tau_0$. Assume that $|b| \ge 100a$. Then
\[
\prob[\tau_0 > 10|b|^{-1}]\le \prob[\tau_0 > (10|a-b|/99)^{-1}].
\]
Using Parts \ref{lem: nhl ai} and \ref{lem: nhl bi} of Lemma \ref{lem: nhl} (and that our assumption implies $10|b|^{-1} \le(10a)^{-1}$), we deduce that
\begin{equation}\label{eq: ep0 1}
    \prob[\tau_0 > 10|b|^{-1}] \le C \exp(-c |b|) \;\text{for any}\; |b| \ge 100 a.
\end{equation}

Assume that $|b| \le a/10$; then
\[
\prob[\cE_{0.\ii}] = \prob[ \tau_0 < (10a)^{-1}] \le \prob[\tau_0 < (11|a-b|^{-1}/100)].
\]
By Lemma \ref{lem: nhl}.\ref{lem: nhl c}, we have
\begin{equation}\label{eq: ep0 2}
    \prob[\cE_{0.\ii}] \le C \exp(-c a) \;\text{for any}\; |b| \le  a /10.
\end{equation}

We now claim that
\begin{equation}\label{eq: ep0 3}
    \prob[\cE_1] \le C \exp(-c|b|)\;\text{for any}\; |b| \ge 100a.
\end{equation}
Observe that
\begin{equation}\label{eq: ep0 12}
\begin{split}
    \prob[\cE_1] \le & \prob[\tau_0 = (10a)^{-1}] + \prob[(q(\tau_0) = 2q_0 \;\text{AND}\; \cE_1]\\
    & + \prob[(q(\tau_0) = q_0/2)\;\text{AND}\; \cE_1].
\end{split}
\end{equation}
Inequality \eqref{eq: ep0 1} implies that
\begin{equation}\label{eq: ep0 13}
    \prob[\tau_0 = (10a)^{-1}] \le C \exp(-c|b|) \;\text{for}\; |b| \ge 100a.
\end{equation}
By Lemma \ref{lem: nhl}.\ref{lem: nhl biii}, we have
\begin{equation}\label{eq: ep0 14}
    \prob[q(\tau_0) = 2q_0]  \le C \exp(-c |b|) \;\text{for any}\; b \ge 100a.
\end{equation}
Recall that if we enter Subepoch $0.\ii$ with $q(\tau_0) = q_0 /2 $, then we set $ u = - K q$. We then enter Epoch 1 if and only if we encounter $|q| = q_0$. Therefore, by Lemma \ref{lem: nhl}.\ref{lem: nhl biii},
\begin{equation}\label{eq: ep0 15}
    \prob[(q(\tau_0) = q_0/2)\;\text{AND}\;\cE_1] \le C \exp(-c|b|)\;\text{for any}\; b \ge 100a.
\end{equation}
Similarly, if we enter Subepoch $0.\ii$ with $q(\tau_0) = 2q_0$, then we set $u = Kq$. We then enter Epoch 1 if and only if we encounter $|q| = 4q_0$. Therefore (again by Lemma \ref{lem: nhl}.\ref{lem: nhl biii}),
\begin{equation}\label{eq: ep0 16}
    \prob[(q(\tau_0) = 2q_0)\;\text{AND}\; \cE_1] \le C \exp(-c |b|)\;\text{for any}\; b \le -100a.
\end{equation}
By Lemma \ref{lem: nhl}.\ref{lem: nhl aii},
\begin{equation}\label{eq: ep0 17}
    \prob[q(\tau_0) = q_0/2] \le C \exp(-c|b|)\;\text{for any}\; b \le -100a.
\end{equation}
Combining \eqref{eq: ep0 12}-\eqref{eq: ep0 17} proves \eqref{eq: ep0 3}.

The remainder of this section is devoted to showing that
\begin{equation}\label{eq: ep0 11}
\E[\cost_0(b)] \le C \cdot \ecost_\op(b)\;\text{for any}\; b \in \R.
\end{equation}

During Subepoch $0.\I$ we set $u=-q$ and we have $|q| \le C$ with probability 1. Since Subepoch $0.\I$ lasts for at most time $(10a)^{-1}$, we have
\begin{equation}\label{eq: ep0 4}
    \E[\cost_{0.\I}(b)] \le C a^{-1}\;\text{for any}\; b \in \R.
\end{equation}
Additionally, \eqref{eq: ep0 1} implies that
\begin{equation}\label{eq: ep0 5}
    \E[\cost_{0.\I}(b)] \le C |b|^{-1}\;\text{for any}\; |b| \ge 100a.
\end{equation}
Combining \eqref{eq: ep0 4} and \eqref{eq: ep0 5} with Lemma \ref{lem: j0} gives
\begin{equation}\label{eq: ep0 6}
    \E[\cost_{0.\I}(b)] \le C \cdot \ecost_\op(b)\;\text{for any}\; b \in \R.
\end{equation}

During Subepoch 0.ii (if it occurs) we again have $|q| \le C$ and thus $|u| \le C$ (since $|u| = K |q|$). Therefore \eqref{eq: ep0 2} implies that
\begin{equation}\label{eq: ep0 7}
    \E[\cost_{0.\ii}(b)] \le C a^{-1}\;\text{for any}\; |b|\le a /10.
\end{equation}
Recall that if we enter Subepoch 0.ii with $q(\tau_0) = q_0/2$, then during Subepoch 0.ii we set $u=-K q$.  By Remark \ref{rmk: cg}.\ref{rmk: cg a}, we therefore have
\begin{equation}\label{eq: ep0 8}
    \E[\cost_{0.\ii}(b)\cdot \mathbbm{1}_{q(\tau_0) = q_0/2}] \le C |b|^{-1}\;\text{for any}\; b \ge a /10
\end{equation}
provided $K$ is large enough. 

Assume that $b \le - a/10$. Then $|b|^{-1} > K|a-bK|^{-1}$, so by Lemma \ref{lem: nhl}.\ref{lem: nhl ai} we have
\[
\E\big[\mathbbm{1}_{q(\tau_0)=q_0/2}\cdot \prob[(\tilde{\tau}_0>|b|^{-1} )| \cF_{\tau_0}]\big] \le C\exp(-c|b|)
\]
provided $K$ is sufficiently large. Since
\begin{multline*}
\E[\cost_{0.\ii}\cdot \mathbbm{1}_{q(\tau_0)=q_0/2}]\\ \le C |b|^{-1} + C \cdot\E\big[\mathbbm{1}_{q(\tau_0)=q_0/2}\cdot \prob[(\tilde{\tau}_0>|b|^{-1}) | \cF_{\tau_0}]\big],
\end{multline*}
we deduce that
\begin{equation}\label{eq: ep0 9}
    \E[\cost_{0.\ii}(b) \cdot \mathbbm{1}_{q(\tau_0) = q_0/2}] \le C |b|^{-1} \;\text{for any}\; b \le - a/10.
\end{equation}

If we enter Subepoch 0.ii with $q(\tau_0) = 2q_0$, then we set $u = K q$ during Subepoch 0.ii. As in the proofs of \eqref{eq: ep0 8} and \eqref{eq: ep0 9}, we use Remark \ref{rmk: cg}.\ref{rmk: cg a} and Lemma \ref{lem: nhl}.\ref{lem: nhl ai} to deduce that
\begin{equation}\label{eq: ep0 10}
    \E[\cost_{0.\ii}(b) \cdot \mathbbm{1}_{q(\tau_0) = 2q_0}] \le C |b|^{-1}\;\text{for any}\; |b| \ge a/10.
\end{equation}
Combining \eqref{eq: ep0 7}--\eqref{eq: ep0 10} with Lemma \ref{lem: j0} and \eqref{eq: ep0 6} proves \eqref{eq: ep0 11}.

\subsection{Probability estimates for Epoch $\mu$ (for $1 \le \mu \le \nu^*$)}\label{sec: nu}

We let $m_0, n_0$ be sufficiently large positive integers so that the conclusion of Lemma \ref{lem: BKPL} holds. Note that we can take $m_0, n_0$ to be absolute constants. We fix $m_0, n_0$ for the remainder of Section \ref{sec: def}.

Recall that we defined the event $\cE_{\nu,m}$ as
\[
\cE_{\nu,m} = \{ \cE_\nu \text{ AND } (\cE_{\mu.\ii}\;\text{for exactly }m\text{ integers } \mu \text{ with } 0 \le \mu < \nu)\}.
\]
The purpose of the present section is to prove the following estimates:
\begin{enumerate}[label={(P.\Roman*)}, ref={P.\Roman*}]
    \item\label{eq: nu 8} For any integers $\mu,m$ satisfying $2 \le \mu \le (\nu^*+1)$ and $0 \le m \le \mu$, we have
    \begin{equation*}
    \prob[\cE_{\mu,m}] \le C \exp(-c 4^{2m}e^\mu |b|)\;\text{for any}\; |b| \ge a e^{-\mu+n_0+1}.
    \end{equation*}
    \item\label{eq: nu 11.1} For any integers $\mu,m$ satisfying $2 \le \mu \le (\nu^*+1)$ and $2 \le m \le \mu$ we have
    \begin{equation*}
    \prob[\cE_{\mu,m}] \le C \exp(-c 4^{2m}a)\;\text{for any}\; |b| \le a e^{-\mu-m_0+1}.
    \end{equation*}
    \item\label{eq: nu 10} For any integers $\mu,m$ satisfying $1 \le \mu \le \nu^*$ and $0 \le m \le \mu$, we have
    \begin{equation*}
    \prob[\cE_{\mu.\ii}\;\text{AND}\; \cE_{\mu,m}] \le C \exp(-c 4^{2m} a)\;\text{for any}\; |b| \le a e^{-\mu-m_0}.
    \end{equation*}
\end{enumerate}

We first prove \eqref{eq: nu 10}. Let $1 \le \mu \le \nu^*$ and $0 \le m \le \mu$. Recall that if $t_\mu < T$, i.e., if Subepoch $\mu.\I$ occurs, then:
\begin{itemize}
\item Subepoch $\mu.\I$ lasts until time $\tau_\mu$, where $\tau_\mu$ is equal to the first time $t \in (t_\mu, \min\{T,(t_\mu+\tau)\})$ for which $|q(t) - q(t_\mu)| \ge |q(t_\mu)| (a\tau)^{1/2}$ if such a time exists and equal to $\min\{T,(t_\mu+\tau)\}$ if no such time exists. 
    \item We set $u = e^\mu q(t_\mu) (a\tau)^{-1/2}$ during Subepoch $\mu.\I$.
    \item If $\tau_\mu < \min\{T, (t_\mu +\tau)\}$, then at time $\tau_\mu$ we enter Subepoch $\mu.\ii$.
    \item If $\tau_\mu =  (t_\mu +\tau) < T$, then at time $\tau_\mu$ we enter Epoch $(\mu+1)$.
\end{itemize}
Therefore, Lemma \ref{lem: BKPL} implies that
\[
\E[\mathbbm{1}_{\cE_{\mu.\ii}} | \cF_{t_\mu}] \le C \exp(-c |q(t_\mu)|^2 a) \;\text{for any}\; |b| \le a e^{-\mu-m_0}.
\]
Recall (see \eqref{eq: q bound}) that if the event $\cE_{\mu,m}$ occurs, then we have $|q(t_\mu)| \ge c 4^m$ with probability 1. Since the event $\cE_{\mu,m}$ is $\cF_{t_\mu}$-measurable, we have
\begin{align*}
\E[\mathbbm{1}_{\cE_{\mu,m}} \mathbbm{1}_{\cE_{\mu.\ii}}] &= \E \big[ \mathbbm{1}_{\cE_{\mu,m}} \cdot \E[\mathbbm{1}_{\cE_{\mu.\ii}}|\cF_{t_\mu}]\big]\\
& \le C \E[\mathbbm{1}_{\cE_{\mu,m}} \exp(-c |q(t_\mu)|^2a)] \\
& \le C \exp(-c 4^{2m}a)\;\text{for any}\; |b| \le a e^{-\mu-m_0}.
\end{align*}
This proves inequality \eqref{eq: nu 10}.

We now prove \eqref{eq: nu 8}. Let $\mu,m$ be integers satisfying $2 \le \mu \le (\nu^*+1)$ and $0 \le m \le \mu$.  Observe that
\begin{equation}\label{eq: 3.22}
\begin{split}
\prob[\cE_{\mu,m}] \le& \E[\mathbbm{1}_{\cE_{(\mu-1),(m-1)}}\mathbbm{1}_{\cE_{\mu,m}}]\cdot \mathbbm{1}_{m>0} \\
&+ \E[\mathbbm{1}_{\cE_{(\mu-1).\ii}^c} \mathbbm{1}_{\cE_{(\mu-1),m}}\mathbbm{1}_{t_{\mu-1} < (T-\tau)}]\cdot \mathbbm{1}_{0 \le m \le (\mu-1)}.
\end{split}
\end{equation}

Assume that $0 \le m \le (\mu-1)$. By Lemma \ref{lem: BKPL}, we have
\begin{equation}\label{eq: n1}
\begin{split}
    \mathbbm{1}_{t_{\mu-1}<T-\tau}&\cdot\E[\mathbbm{1}_{\cE_{(\mu-1).\ii}^c} | \cF_{t_{\mu-1}}] \\
    &\qquad \le C \exp(- c |q(t_{\mu-1})|^2 b^2 e^{2\mu} / a)\;\text{for any}\; |b| \ge a e^{-\mu + n_0 + 1}.
\end{split}
\end{equation}
Since the events $\cE_{(\mu-1),m}$ and $t_{\mu-1}<(T-\tau)$ are $\cF_{t_{\mu-1}}$-measurable, and since $\cE_{(\mu-1),m}$ implies that $|q(t_{\mu-1})| > c 4^m$ with probability 1, we use \eqref{eq: n1} to get
\begin{equation}\label{eq: 3.23}
    \begin{split}
        \E[\mathbbm{1}_{\cE_{(\mu-1).\ii}^c}\mathbbm{1}_{\cE_{(\mu-1),m}}&\mathbbm{1}_{t_{\mu-1}<(T-\tau)}] \\
        = & \E\big[ \mathbbm{1}_{\cE_{(\mu-1),m}}\mathbbm{1}_{t_{\mu-1}<(T-\tau)}\cdot \E[\mathbbm{1}_{\cE_{(\mu-1).\ii}^c} | \cF_{t_{\mu-1}}] \big]\\
         \le & C \exp(-c 4^{2m}b^2 e^{2\mu}/a)\;\text{for any}\; |b| \ge a e^{-\mu + n_0 +1}.
    \end{split}
\end{equation}

Now assume that $0 < m \le \mu$. Recall that if $\tit_{\mu-1} < T$, i.e., if Subepoch $(\mu-1).\ii$ occurs, then
\begin{itemize}
    \item we set $u = K e^{\mu-1}q$ during Subepoch $(\mu-1).\ii$, and
    \item  Subepoch $(\mu-1).\iii$ occurs if and only if we encounter $|q| = 2 |q(\tilde{t}_{\mu-1})|$ during Subepoch $(\mu-1).\ii$.
\end{itemize}
Similarly, if $\hat{t}_{\mu-1} < T$, i.e., if Subepoch $(\mu-1).\iii$ occurs, then
\begin{itemize}
    \item we set $u =- K e^{\mu-1}q$ during Subepoch $(\mu-1).\iii$, and
    \item  Epoch $\mu$ occurs if and only if we encounter $|q| = 2 |q(\hat{t}_{\mu-1})|$ during Subepoch $(\mu-1).\iii$.
\end{itemize}
We can therefore use Lemma \ref{lem: nhl}.\ref{lem: nhl biii} to get
\begin{equation}\label{eq: n2}
\E[\mathbbm{1}_{\cE_{(\mu-1).\iii}} | \cF_{\tilde{t}_{\mu-1}}] \le C \exp(-c |q(\tilde{t}_{\mu-1})|^2 e^{\mu}|b|)\;\text{for any}\; b \le - a e^{-\mu + n_0 + 1}
\end{equation}
and
\begin{equation}\label{eq: n3}
\E[\mathbbm{1}_{\cE_\mu} |  \cF_{\hat{t}_{\mu-1}}]
\le C \exp(-c |q(\hat{t}_{\mu-1})|^2 e^\mu |b| )\;\text{for any}\; b \ge a e^{-\mu + n_0 +  1}.
\end{equation}
Since
\[
\mathbbm{1}_{\cE_{\mu,m}}\mathbbm{1}_{\cE_{(\mu-1),(m-1)}} = \mathbbm{1}_{\cE_\mu}\mathbbm{1}_{\cE_{(\mu-1).\iii}}\mathbbm{1}_{\cE_{(\mu-1).\ii}}\mathbbm{1}_{\cE_{(\mu-1),(m-1)}},
\]
we use \eqref{eq: n2}, \eqref{eq: n3}, and \eqref{eq: q bound} to deduce that
\begin{equation}\label{eq: 3.26}
    \E[\mathbbm{1}_{\cE_{\mu,m}} \mathbbm{1}_{\cE_{(\mu-1),(m-1)}}] \le C \exp(-c 4^{2m} e^\mu |b|) \;\text{for any}\; |b| \ge a e^{-\mu+n_0+1}.
\end{equation}
Combining \eqref{eq: 3.22}, \eqref{eq: 3.23}, \eqref{eq: 3.26} proves \eqref{eq: nu 8}.

We now prove \eqref{eq: nu 11.1}. We claim that for any $2\le \mu\le (\nu^*+1)$ and $2 \le m \le \mu$ we have
\begin{equation}\label{eq: nu 11}
    \prob[\cE_{\mu,m}] \le C (\mu-m+1) \exp(-c4^{2m}a)\;\text{for any}\; |b| \le a  e^{-\mu-m_0+1}.
\end{equation}
Since $\nu^* \le C a$, \eqref{eq: nu 11} implies \eqref{eq: nu 11.1}.

We fix $m$ with $2 \le m \le (\nu^*+1)$. Observe that for any $\mu$ with $m \le \mu \le (\nu^*+1)$ we have
\begin{equation}\label{eq: nu 9}
\begin{split}
    \prob[\cE_{\mu,m}] \le \prob[\cE_{(\mu-1).\ii}\;\text{AND}&\;\cE_{(\mu-1),(m-1)}]\\
    &+\prob[\cE_{(\mu-1),m}]\cdot \mathbbm{1}_{\mu \ge (m+1)}.
\end{split}
\end{equation}
We now show by induction that \eqref{eq: nu 11} holds for any such $\mu$.

\underline{Base Case}: Inequality \eqref{eq: nu 9} and \eqref{eq: nu 10} imply that
\[
\prob[\cE_{m,m}] \le C \exp(-c 4^{2m}a)\;\text{for any}\; |b| \le a  e^{-m-m_0+1}.
\]

\underline{Induction Step}: Suppose that
\begin{equation}\label{eq: nu 11.5}
\prob[\cE_{(\mu-1),m}] \le C (\mu-m) \exp(-c 4^{2m}a)\;\text{for any}\; |b| \le a  e^{-\mu-m_0+2}
\end{equation}
for some $\mu$ with $m < \mu \le (\nu^*+1)$. Combining \eqref{eq: nu 11.5} with \eqref{eq: nu 9} and \eqref{eq: nu 10} gives
\[
\prob[\cE_{\mu,m}] \le C(\mu-m+1) \exp(-c4^{2m}a)\;\text{for any}\; |b| \le a e^{-\mu-m_0+1}.
\]
This proves \eqref{eq: nu 11}.

\subsection{Expected cost during Epoch $\mu$ (for $1 \le \mu \le \nu^*$)}

Recall that if $t_\mu < T$, i.e., if Subepoch $\mu.\I$ occurs, then (with probability 1):
\begin{itemize}
    \item Subepoch $\mu.\I$ lasts for at most time $\tau$.
    \item During Subepoch $\mu.\I$ we have $|q| \le C |q(t_\mu)|$ and $|u| = e^\mu(a\tau)^{-1/2} |q(t_\mu)|$.
\end{itemize}
Taking $\tau < a^{-1}$ (recall that $\tau$ is chosen to be sufficiently small depending on $a$), we get
\begin{equation}\label{eq: nu 14}
    \E[\cost_{\mu.\I}(b) | \cF_{t_\mu}] \le Ce^{2\mu}|q(t_\mu)|^2 a^{-1} \;\text{for any}\; b \in \R.
\end{equation}

If $\tit_\mu < T$, i.e., if Subepoch $\mu.\ii$ occurs, then during Subepoch $\mu.\ii$ we set $u = K e^\mu q$ and so (with probability 1) we have $|q| \le 2 |q(\tilde{t}_\mu)|$ and $|u| \le C |q(\tilde{t}_\mu)|e^\mu$. Similarly, if $\htt_\mu <T$, then during Subepoch $\mu.\iii$ we set $u = -K e^\mu q$ and so (with probability 1) we have $|q| \le 2 | q(\htt_\mu)|$ and $|u| \le C |q(\htt_\mu)| e^\mu$. Therefore
\begin{align}
    &\E[\cost_{\mu.\ii}(b) | \cF_{\tilde{t}_\mu}] \le C e^{2\mu} |q(\tilde{t}_\mu)|^2\;\text{for any}\; b \in \R,\label{eq: nu 15} \\
    &\E[\cost_{\mu.\iii}(b) | \cF_{\hat{t}_\mu}] \le C e^{2\mu} |q(\hat{t}_\mu)|^2 \;\text{for any}\; b \in \R.\label{eq: nu 15.5}
\end{align}

It will be useful to have sharper versions of \eqref{eq: nu 15}, \eqref{eq: nu 15.5}.  We claim that for any $|b| \ge a e^{-\mu-m_0}$, we have
\begin{align}
& \E[\cost_{\mu.\ii}(b) | \cF_{\tilde{t}_\mu}] \le C e^{2\mu} a^{-1}|q(\tilde{t}_\mu)|^2,\label{eq: nu 25}\\
& \E[\cost_{\mu.\iii}(b) | \cF_{\hat{t}_\mu}] \le C e^{2\mu} a^{-1}|q(\hat{t}_\mu)|^2.\label{eq: nu 26}
\end{align}
We prove \eqref{eq: nu 25}; the proof of \eqref{eq: nu 26} is essentially identical.

We condition on the history up to time $\tilde{t}_\mu < T$ (if $\tilde{t}_\mu = T$, then $\cost_{\mu.\ii}=0$ and thus \eqref{eq: nu 25} is trivial ). Recall that Subepoch $\mu.\ii$ ends before time $T$ if and only if we encounter $|q| = 2 | q(\tilde{t}_\mu)|$. Provided $K$ is large enough depending on the constant $m_0$, we use Remark \ref{rmk: cg}.\ref{rmk: cg a} to establish \eqref{eq: nu 25} for $b \le -a e^{-\mu-m_0}$. Lemma \ref{lem: nhl}.\ref{lem: nhl ai} implies \eqref{eq: nu 25} for any $b \ge 0$ (and in particular for any $b \ge a e^{-\mu-m_0}$). This completes the proof of \eqref{eq: nu 25}.

\subsubsection{The Epochs in which our guess for $b$ is too small}

Assume that:
\begin{itemize}
    \item $\nu$ is an integer with $1 \le (\nu-m_0) \le \nu^*$,
    \item $|b| \le a e^{-\nu}$.
\end{itemize}

By \eqref{eq: nu 11.1}, for any $2 \le \mu \le (\nu-m_0)$ and $2 \le m \le \mu$ we have
\begin{equation}\label{eq: uc 2}
\prob[\cE_{\mu,m}] \le C \exp(-c 4^{2m}a).
\end{equation}
Combining this with \eqref{eq: nu 14} and \eqref{eq: q bound} gives
\[
\E[\cost_{\mu.\I}(b)\cdot \mathbbm{1}_{\cE_{\mu,m}}] \le C e^{2\mu}a^{-1} 4^{-2m};
\]
therefore
\begin{equation}\label{eq: uc 4}
\E[\cost_{\mu.\I}(b)] = \sum_{m=0}^\mu \E[\cost_{\mu.\I}(b) \cdot \mathbbm{1}_{\cE_{\mu,m}}]
 \le Ce^{2\mu} a^{-1}
\end{equation}
for any $2\le \mu \le (\nu-m_0)$. Observe that \eqref{eq: nu 14} also implies that
\begin{equation}\label{eq: uc 5}
    \E[\cost_{1.\I}(b)] \le C a^{-1};
\end{equation}
combining \eqref{eq: uc 4} and \eqref{eq: uc 5} gives
\begin{equation}\label{eq: uc 6}
\sum_{\mu=1}^{\nu-m_0 } \E[\cost_{\mu.\I}(b)] \le C e^{2\nu}a^{-1}.
\end{equation}

Now observe that, by \eqref{eq: nu 10}, for any $1 \le \mu \le (\nu-m_0)$ and $0 \le m \le \mu$ we have
\begin{equation}\label{eq: uc 3}    \prob[\cE_{\mu.\ii}\;\text{AND}\;\cE_{\mu,m}]\le C \exp(-c 4^{2m}a).
\end{equation}
Recall that $\cE_{\mu.\iii}$ occurs only if $\cE_{\mu.\ii}$ occurs. Therefore, combining \eqref{eq: nu 15}, \eqref{eq: nu 15.5}, \eqref{eq: q bound}, and \eqref{eq: uc 3} gives
\[
\E[(\cost_{\mu.\ii}(b)+\cost_{\mu.\iii}(b))\cdot \mathbbm{1}_{\{\cE_{\mu,m}\;\text{AND}\; \cE_{\mu.\ii}\}}] \le C e^{2\mu} a^{-1} 4^{-2m},
\]
and we deduce that
\[
\E[\cost_{\mu.\ii}(b)+
\cost_{\mu.\iii}(b)] 
    \le  Ce^{2\mu} a^{-1}
\]
for any $1 \le \mu \le (\nu-m_0)$. Summing over $\mu$, we get
\begin{equation}\label{eq: uc 7}
\sum_{\mu=1}^{\nu-m_0} \E[\cost_{\mu.\ii}(b)+\cost_{\mu.\iii}(b)] \le Ce^{2\nu } a^{-1}.
\end{equation}
Combining \eqref{eq: uc 6} and \eqref{eq: uc 7} gives
\begin{equation}\label{eq: uc 8}
\sum_{\mu=1}^{\nu-m_0} \E[\cost_\mu(b)]\le Ce^{2\nu}a^{-1}\;\text{for any}\;  |b| \le a e^{-\nu}.
\end{equation}

\subsubsection{The Epochs in which our guess for $b$ is too large}

Assume that:
\begin{itemize}
    \item $\nu$ is an integer with $m_0 \le \nu \le (\nu^*-n_0-2)$,
    \item $|b| \ge a e^{-\nu-1}$.
\end{itemize}

By \eqref{eq: nu 8}, for any $\mu$ satisfying
\[
(\nu+n_0+2) \le \mu \le \nu^*
\]
and for any $0 \le m \le \mu$ we have
\begin{equation}\label{eq: oc 2}
    \prob[\cE_{\mu,m}] \le C \exp(-c 4^{2m}e^{\mu}|b|).
\end{equation}
Combining \eqref{eq: oc 2} with \eqref{eq: nu 14}, \eqref{eq: nu 15}, \eqref{eq: nu 15.5}, and \eqref{eq: q bound} gives (for any such $\mu$)
\begin{equation*}
        \E[\cost_\mu(b)] = \sum_{m=0}^\mu \E[\cost_\mu(b)\cdot \mathbbm{1}_{\cE_{\mu,m}}] \le \frac{C}{e^\mu |b|^3}.
\end{equation*}
Summing over $\mu$, we get
\begin{equation}\label{eq: oc 2.5b}
    \sum_{\mu=(\nu+n_0+2)}^{\nu^*}\E[\cost_\mu(b)]\le \frac{C}{e^\nu |b|^3} \;\text{for any}\; |b|\ge a e^{-\nu-1};
\end{equation}
this implies
\begin{equation}\label{eq: oc 2.5}
    \sum_{\mu=(\nu+n_0+2)}^{\nu^*}\E[\cost_\mu(b)]\le \frac{Ce^{2\nu}}{a} \;\text{for any}\; |b|\ge a e^{-\nu-1}.
\end{equation}

\subsubsection{The Epochs in which our guess for $b$ is accurate}

Assume that:
\begin{itemize}
    \item $\nu$ is an integer with $(m_0+1) \le \nu < (\nu^*+m_0)$,
    \item $ae^{-\nu-1}\le |b| \le a e^{-\nu}$,
    \item $(\nu-m_0+1) \le \mu \le \min \{\nu^*, (\nu+n_0+1)\}$.
\end{itemize}

Inequalities \eqref{eq: nu 25} and \eqref{eq: nu 26} imply
\begin{align}
    & \E[\cost_{\mu.\ii}(b) | \cF_{\tilde{t}_\mu}] \le Ce^{2\mu} a^{-1} | q(\tilde{t}_\mu)|^2,\label{eq: nu 25b}&\\
    & \E[\cost_{\mu.\iii}(b) | \cF_{\hat{t}_\mu}] \le Ce^{2\mu} {a}^{-1} | q(\hat{t}_\mu)|^2.\label{eq: nu 26b}
\end{align}

Now observe that, by \eqref{eq: nu 11.1}, we have
\begin{equation}\label{eq: ac 1}
    \prob[\cE_\mu \;\text{AND}\; \cE_{(\nu-m_0+1),m}]\le\prob[\cE_{(\nu-m_0+1),m}] \le C \exp(-c 4^{2m}a)
\end{equation}
for any $2 \le m \le (\nu-m_0+1)$. Combining \eqref{eq: ac 1} and \eqref{eq: q bound} with \eqref{eq: nu 14}, \eqref{eq: nu 25b}, \eqref{eq: nu 26b} gives (respectively)
\[
\E[\cost_{\mu.j}(b)] \le C e^{2\mu}a^{-1}\;\text{for}\;j = \I, \ii,\iii.
\]
Summing over $\mu$ gives
\begin{equation}\label{eq: ac 3}
   \sum_{\mu = (\nu-m_0+1)}^{\min\{\nu^*,(\nu+n_0+1)\}} \E[\cost_\mu(b)]  \le C e^{2\nu}{a}^{-1} \;\text{for any}\; a e^{-\nu-1} \le |b| \le a e^{-\nu}.
\end{equation}

\subsubsection{Putting it all together}

Taking $\nu = (\nu^*+m_0)$ in \eqref{eq: uc 8} and using Lemma \ref{lem: j0} (recall that $\nu^*\le aT$), we deduce
\begin{equation}\label{eq: uc 10}
\sum_{\mu=1}^{\nu^*}\E[\cost_\mu(b)] \le C \cdot \ecost_\op(b) \;\text{for any}\;  |b| \le a e^{-\nu^*-m_0}.
\end{equation}
Setting $\nu = m_0$ in \eqref{eq: oc 2.5} gives
\begin{equation}\label{eq: oc 4}
    \sum_{\mu = m_0+n_0+2}^{\nu^*} \E[\cost_\mu(b)] \le Ca^{-1}\;\text{for any}\; |b| \ge a e^{-m_0-1}.
\end{equation}

By \eqref{eq: q bound}, we have the following for any $1 \le \mu \le (m_0 + n_0 + 1)$: If $t_\mu < T$ (resp. $\tit_\mu<T$, $\htt_\mu <T$), then $|q(t_\mu)| < C$ (resp. $|q(\tit_\mu)| < C$, $|q(\htt_\mu)| < C$) with probability 1. Therefore, \eqref{eq: nu 14}, \eqref{eq: nu 25}, and \eqref{eq: nu 26} imply that for any $|b| \ge ae^{-m_0-1}$ we have
\begin{equation}\label{eq: oc b2}
\E[\cost_\mu(b) | \cF_{t_\mu}] \le Ca^{-1}\;\text{for any}\; 1 \le \mu \le (m_0 + n_0+1).
\end{equation}
Therefore
\begin{equation}\label{eq: oc 5}
    \sum_{\mu=1}^{m_0+n_0+1} \E[\cost_{\mu}(b)] \le C a^{-1} \;\text{for any}\; |b| \ge a e^{-m_0-1}.
\end{equation}
In particular, \eqref{eq: oc 4}, \eqref{eq: oc 5} and Lemma \ref{lem: j0} give
\begin{equation}\label{eq: oc 6}
\sum_{\mu=1}^{\nu^*} \E[\cost_{\mu}(b)] \le C\cdot \ecost_\op(b) \;\text{for any}\; 100a \ge |b| \ge a e^{-m_0-1}.
\end{equation}

Recall inequality \eqref{eq: ep0 3}:
\[
\prob[\cE_1] \le C \exp(-c|b|) \;\text{for any}\; |b| \ge 100a.
\]
Combining this with \eqref{eq: oc 2.5b}, \eqref{eq: oc b2}, and Lemma \ref{lem: j0} gives
\begin{equation}\label{eq: oc 8}
\sum_{\mu=1}^{\nu^*} \E[\cost_{\mu}(b)] \le C\cdot \ecost_\op(b) \;\text{for any}\; |b| \ge 100a.
\end{equation}

For any $(\nu^* - n_0 - 1) \le \nu < (\nu^* + m_0)$, we use \eqref{eq: uc 8}, \eqref{eq: ac 3}, and Lemma \ref{lem: j0} to get
\begin{equation}\label{eq: oc b1}
    \sum_{\mu=1}^{\nu^*} \E[\cost_{\mu}(b)] \le C \cdot  \ecost_\op(b) \;\text{for any}\; a e^{-\nu-1}\le |b| \le a e^{-\nu}.
\end{equation}
Similarly, we use \eqref{eq: uc 8}, \eqref{eq: oc 2.5}, \eqref{eq: ac 3} and Lemma \ref{lem: j0} to deduce that \eqref{eq: oc b1} also holds for any $(m_0 + 1) \le \nu \le(\nu^*-n_0-2)$. We conclude that
\begin{equation}\label{eq: ac 4}
\sum_{\mu=1}^{\nu^*} \E[\cost_{\mu}(b)] \le C \cdot  \ecost_\op(b) \;\text{for any}\; a e^{-\nu^*-m_0} \le |b| \le a e^{-m_0-1}.
\end{equation}

Combining \eqref{eq: uc 10}, \eqref{eq: oc 6}, \eqref{eq: oc 8}, and \eqref{eq: ac 4} gives
\begin{equation}\label{eq: nu main}
    \sum_{\mu=1}^{\nu^*} \E[\cost_{\mu}(b)] \le C \cdot \ecost_\op(b)\;\text{for any}\; b \in \R.
\end{equation}

\subsection{Performance during the Final Epoch}

Recall that during Epoch $(\nu^*+1)$ we set $u=0$. Therefore, by Remark \ref{rmk: cg}.\ref{rmk: cg c}, we have
\begin{equation}\label{eq: star 1}
    \E[\cost_{\nu^*+1}(b) | \cF_{t_{(\nu^*+1)}}] \le \frac{e^{2aT}}{2a}[|q(t_{(\nu^*+1)})|^2+T].
\end{equation}
For any $2 \le m \le (\nu^*-m_0-n_0)$, \eqref{eq: nu 11.1} implies that
\begin{multline}\label{eq: star 2}
\prob[\cE_{\nu^*+1}\;\text{AND}\; \cE_{(\nu^*-m_0-n_0),m}]\\ \le C \exp(-c 4^{2m}a) \;\text{for any}\; |b| \le a  e^{-\nu^*+n_0+1}.
\end{multline}
Combining \eqref{eq: star 1}, \eqref{eq: star 2}, and \eqref{eq: q bound} gives
\begin{align*}
\E[\cost_{\nu^*+1}(b) \cdot \mathbbm{1}_{\cE_{(\nu^*-m_0 - n_0),m}}]
= & \E\big[ \mathbbm{1}_{\cE_{(\nu^*-m_0 - n_0),m}} \cdot \E[\cost_{\nu^*+1}(b) | \cF_{t_{(\nu^*+1)}}] \big]\\
\le & C e^{2aT} a^{-1} 4^{2m} \cdot\E[\mathbbm{1}_{\cE_{\nu^*+1}}\cdot \mathbbm{1}_{\cE_{(\nu^*-m_0 - n_0),m}}]\\
\le & C e^{2aT} a^{-1} 4^{-2m}\;\text{for any}\; |b| \le a e^{-\nu^*-n_0+1}.
\end{align*}
Observe that
\[
\mathbbm{1}_{\cE_{\nu^*+1}} = \sum_{m=0}^{\nu^*-m_0-n_0} \mathbbm{1}_{\cE_{\nu^*+1}}\cdot\mathbbm{1}_{\cE_{(\nu^*-m_0-n_0),m}};
\]
therefore (since $\cost_{\nu^*+1} = \mathbbm{1}_{\cE_{\nu^*+1}}\cdot \cost_{\nu^*+1}$)
\begin{align*}
    \E[\cost_{\nu^*+1}(b)] \le  C e^{2aT}a^{-1}\;\text{for any}\; |b| \le a  e^{-\nu^*+n_0+1}.
\end{align*}
Combining this with Lemma \ref{lem: j0} (recall that $\nu^* = \lfloor aT \rfloor$), we get
\begin{equation}\label{eq: star 3}
\E[\cost_{\nu^*+1}(b)] \le C \cdot \ecost_\op(b)\;\text{for any}\; |b| \le a  e^{-\nu^*+n_0+1}.
\end{equation}

Let $ -1 \le \nu \le (\nu^* - n_0 - 1)$ until further notice. Use \eqref{eq: nu 8} to get that
\begin{equation*}
    \prob[\cE_{(\nu^*+1),m}] \le C \exp(-c 4^{2m}e^{(\nu^*-\nu)} a )\;\text{for any}\; |b| \ge a e^{-\nu-1},
\end{equation*}
where $0 \le m \le (\nu^*+1)$. Combining this with \eqref{eq: star 1} and \eqref{eq: q bound} gives
\begin{align*}
    \E[\cost_{\nu^*+1}(b)\cdot \mathbbm{1}_{\cE_{(\nu^*+1),m}}]  \le C e^{2(aT-\nu^*)}e^{2\nu}a^{-1}4^{-2m} \;\text{for any}\;|b|\ge a e^{-\nu-1}.
\end{align*}
Recall that $\nu^* = \lfloor aT \rfloor$, so $e^{2(aT-\nu^*)} \le C$. We deduce that
\begin{equation}\label{eq: star 4}
    \E[\cost_{\nu^*+1}(b)] \le C e^{2\nu}a^{-1}\;\text{for any}\; |b| \ge a e^{-\nu-1};
\end{equation}
combining this with Lemma \ref{lem: j0} gives
\begin{equation}\label{eq: star 5}
\E[\cost_{\nu^*+1}(b)] \le C \cdot \ecost_\op(b)\;\text{for any}\;  a  e^{-\nu-1} \le |b| \le a  e^{-\nu},
\end{equation}
where $-1 \le \nu \le (\nu^*-n_0-1)$, and (setting $\nu=-1$ in \eqref{eq: star 4})
\begin{equation}\label{eq: star 7.5}
\E[\cost_{\nu^*+1}(b)]\le C \cdot \ecost_\op(b) \;\text{for any}\; a \le |b| \le 100 a .
\end{equation}

Since \eqref{eq: star 5} holds for any $-1 \le \nu \le (\nu^*-n_0-1)$, we combine it with \eqref{eq: star 3} to get
\begin{equation}\label{eq: star 6}
    \E[\cost_{\nu^*+1}(b)] \le C \cdot \ecost_\op(b)\;\text{for any}\; |b| \le a .
\end{equation}

By \eqref{eq: nu 8}, we have for any $0 \le m \le (\nu^*+1)$ that
\[
\prob[\cE_{(\nu^*+1),m}] \le C \exp(-c 4^{2m} e^{\nu^*} |b|)\;\text{for any}\; |b| \ge 100 a.
\]
Combining this with \eqref{eq: star 1} and \eqref{eq: q bound} gives
\begin{equation}
    \E[\cost_{\nu^*+1}(b)\cdot \mathbbm{1}_{\cE_{(\nu^*+1),m}}] \le C  4^{-2m}|b|\;\text{for any}\; |b| \ge 100 a;
\end{equation}
we conclude that
\[
\E[\cost_{\nu^*+1}(b)] \le C |b|^{-1}\;\text{for any}\; |b| \ge 100a.
\]
By Lemma \ref{lem: j0}, we have
\[
\E[\cost_{\nu^*+1}(b)] \le C \cdot \ecost_\op(b) \;\text{for any}\; |b| \ge 100a.
\]
Combining this with \eqref{eq: star 7.5} and \eqref{eq: star 6} gives
\begin{equation}\label{eq: star main}
    \E[\cost_{\nu^*+1}(b)] \le C \cdot \ecost_\op(b)\;\text{for any}\; b\in \R.
\end{equation}

Together, \eqref{eq: ep0 11}, \eqref{eq: nu main}, and \eqref{eq: star main} prove Theorem \ref{thm: main} for $a \ge A$.

\section{The strategy for bounded $a$}\label{sec: bdd}

Throughout Section \ref{sec: bdd} we assume that 
\[
|a| \le A.
\]
In a moment we will define the strategy $\br$ in this case. Throughout this section we write $q,u$ for the particle trajectories and control variables $q^\br, u^\br$.

We now define the strategy $\br$; it is divided into three Epochs, only the first of which is guaranteed to occur. At time 0 we enter Epoch 0.

\underline{Epoch $0$}: During Epoch $0$ we set $u(t) = -q(t).$ Epoch $0$ lasts from time $0$ until time $t_1$, where $t_1$ is equal to the first time $t \in(0,T)$ for which $|q(t)| = 2 q_0$ if such a time exists and is equal to $T$ if no such time exists. If $t_1 = T$, then at time $t_1$ the game ends. If $t_1 < T$, then at time $t_1$ we enter Epoch $1$.

\underline{Epoch $1$}: Suppose that we enter Epoch $1$, i.e., that $t_1 < T$. During Epoch $1$ we set $u(t) = q(t).$ Epoch $1$ lasts from time $t_1$ until time $t_2$, where $t_2$ is equal to the first time $t \in(t_1,T)$ for which $|q(t)| = 4q_0$ if such a time exists and is equal to $T$ if no such time exists. If $t_2 = T$, then at time $t_2$ the game ends. If $t_2 < T$, then at time $t_2$ we enter Epoch $2$.

Note that we have only defined the stopping time $t_2$ in the event that $(t_1 < T)$. For convenience, we define $t_2$ to be equal to $T$ in the event that $(t_1 = T)$.

\underline{Epoch 2}: Note that $t_2$ is equal to the time at which we enter Epoch 2 if such a time exists and is equal to $T$ if no such time exists. If we enter Epoch 2, then we set $u = 0$ until the game ends at time $T$.

This concludes the definition of the strategy $\br$ when $|a| \le A$.

As in Section \ref{sec: def}, we say that Epoch $\nu$ occurs if there exists some time at which we enter Epoch $\nu$ and we let $\cE_\nu$ denote the event that Epoch $\nu$ occurs. We define a random variable $\cost_\nu(b)$ (for $\nu=0,1,2$) to be equal to the cost incurred during Epoch $\nu$ if Epoch $\nu$ occurs and equal to 0 if Epoch $\nu$ does not occur. Observe that
\begin{equation*}
\ecost(\br,b)  = \sum_{\nu=0}^2 \E[\cost_\nu(b)].
\end{equation*}
Our goal in this section is to show that
\begin{equation}\label{eq: bdd main}
    \ecost(\br,b) \le C \cdot \ecost_\op(b)\;\text{for any}\; b \in \R;
\end{equation}
this implies Theorem \ref{thm: main} in the case $|a| \le A$.

 Since we set $u =0$ during Epoch 2, and since we assume that $|a|\le A$,  Remark \ref{rmk: cg}.\ref{rmk: cg b} gives
\begin{equation}\label{eq: bdd b1}
\E[ \cost_2(b) | \cF_{t_2}] \le C \;\text{for any}\; b \in \R.
\end{equation}
Remark \ref{rmk: cg}.\ref{rmk: cg b} also implies that
\begin{equation}\label{eq: bdd b2}
     \E[\cost_0(b)], \E[\cost_1(b)]  \le C \;\text{for any}\; |b| \le 100 A.
\end{equation}
Combining \eqref{eq: bdd b1}, \eqref{eq: bdd b2} with Lemma \ref{lem: j0} gives
\begin{equation}\label{eq: bdd b5}
\ecost(\br,b) \le C \cdot \ecost_\op(b)\; \text{for any} \; |b| \le 100A.
\end{equation}

Now, assume that $b \ge 100A$ and recall that we set $u = - q$ during Epoch 0. We have
\begin{flalign}
    &\E[\cost_0(b)] \le C |b|^{-1}\;\text{by Remark \ref{rmk: cg}.\ref{rmk: cg a}, and}\label{eq: bdd b3}&\\
    &\prob[\cE_2] \le \prob[\cE_1] \le C \exp(-c|b|)\text{ by Lemma \ref{lem: nhl}.\ref{lem: nhl biii}}.
\end{flalign}
Also, since we set $u = q$ during Epoch 1, and since $|q| \le C$ during Epoch 1, we have
\begin{equation*}
    \E[\cost_1(b)] \le C|b|^{-1} + C \cdot \E[\mathbbm{1}_{t_1 < T - 10|b|^{-1}}\cdot \E[\mathbbm{1}_{t_2 > t_1+ |b|^{-1}}|\cF_{t_1}]].
\end{equation*}
Our assumption $b\ge 100A$ implies that $|b|^{-1} \ge 99 |a+b|^{-1}/100$, therefore we can use Lemma \ref{lem: nhl}.\ref{lem: nhl ai} to deduce that
\begin{equation*}
    \E[\mathbbm{1}_{t_1 < T - 10|b|^{-1}}\cdot \E[\mathbbm{1}_{t_2 > t_1+|b|^{-1}}|\cF_{t_1}]] \le C |b|^{-1};
\end{equation*}
we conclude that
\begin{equation}
    \E[\cost_1(b)] \le C |b|^{-1}\;\text{for any}\; b \ge 100A.
\end{equation}

Similarly, if $b \le -100A$, then
\begin{flalign}
&\E[\cost_0(b)]  \le C|b|^{-1}\text{ by Lemma \ref{lem: nhl}.\ref{lem: nhl ai},}& \\
& \E[\cost_1(b)] \le C|b|^{-1} \;\text{by Remark \ref{rmk: cg}.\ref{rmk: cg a} and}, \\
& \prob[\cE_2] \le C \exp(-c|b|)\;\text{by Lemma \ref{lem: nhl}.\ref{lem: nhl biii}}.\label{eq: bdd b4}
\end{flalign}
Combining \eqref{eq: bdd b3}--\eqref{eq: bdd b4} with \eqref{eq: bdd b1} gives
\[
\ecost(\br,b) \le C|b|^{-1}\;\text{for any}\; |b| \ge 100A;
\]
we use Lemma \ref{lem: j0} to deduce that
\begin{equation}\label{eq: bdd b6}
\ecost(\br, b) \le C \cdot \ecost_\op(b)\;\text{for any}\; |b| \ge 100A.
\end{equation}
Combining \eqref{eq: bdd b5}, \eqref{eq: bdd b6} proves \eqref{eq: bdd main}.

 \section{The strategy for large, negative $a$}\label{sec: neg}

Throughout this section we assume that 
\[
a \le - A.
\]
In a moment, we will define the strategy $\br$ in this case. Throughout this section we write $q,u$ for the particle trajectories and control variables $q^\br, u^\br$.

We now define the strategy $\br$. It consists of three epochs, only the first two of which are guaranteed to occur.

\underline{Epoch 0}: Epoch 0 begins at time 0. During Epoch 0 we set $u = - q$. Epoch 0 lasts until time $t_1$, where $t_1$ is equal to the first time $t \in (0, (10|a|)^{-1})$ for which $q(t) = 2 q_0$ or $q(t) = \frac{1}{2} q_0$ if such a time exists and equal to $(10|a|)^{-1}$ if no such time exists. At time $t_1$ we enter Epoch 1. (Recall that $A$ was chosen so that $(10|a|)^{-1} < T$ for any $a \le -A$; see Section \ref{sec: def}.)

\underline{Epoch 1}: If we enter Epoch 1 at time $t_1 = (10|a|)^{-1}$, then we set $u=0$ until the game ends at time $T$.

If we enter Epoch 1 at time $t_1 < (10|a|)^{-1}$, then with probability 1 we have either $q(t_1) = 2q_0$ or $q(t_1) = \frac{1}{2} q_0$.

If $q(t_1) = 2q_0$, then during Epoch 1 we set $u=1000\cdot q$. If $q(t_1) = \frac{1}{2} q_0$, then during Epoch 1 we set $u = -1000\cdot q$. In either case Epoch 1 lasts until time $t_2$, where $t_2$ is equal to the first time $t \in (t_1, T)$ for which $|q(t)| = 4q_0$ if such a time exists and equal to $T$ if no such time exists. If $t_2 < T$, then at time $t_2$ we enter Epoch 2.

We have only defined the stopping time $t_2$ in the event that $(t_1 < T)$. For convenience, we define $t_2$ to be equal to $T$ in the event that $(t_1 = T)$.

\underline{Epoch 2}: Note that $t_2$ is equal to the time at which we enter Epoch 2 if such a time exists and is equal to $T$ if we never enter Epoch 2. If we enter Epoch 2, then we set $u=0$ until the game ends at time $T$. (Note that if $t_1 = (10|a|)^{-1}$, then $t_2 = T$ with probability 1.)

This concludes the definition of the strategy $\br$ in the case $a \le - A$.

As usual, we let $\cE_\nu$ denote the event that Epoch $\nu$ occurs (i.e., the event that there exists some time at which we enter Epoch $\nu$) and we define the random variable $\cost_\nu(b)$ to be equal to the cost incurred during Epoch $\nu$ if Epoch $\nu$ occurs and equal to 0 if Epoch $\nu$ does not occur. Clearly
\begin{equation*}
    \ecost(\br,b) = \sum_{\nu=0}^2 \E[\cost_\nu(b)].
\end{equation*}

We first note that for any $|b| \ge 100|a|$,
\[
\prob[t_1 > 10|b|^{-1}] \le \prob[t_1 > (10|b-a|/99)^{-1}].
\]
Since we set $u = -q$ during Epoch 0, we can use Parts \ref{lem: nhl ai} and \ref{lem: nhl bi} of Lemma \ref{lem: nhl} to get
\begin{equation}\label{eq: neg 5}
    \prob[t_1 > 10|b|^{-1}] \le C \exp(-c |b|)\;\text{for any}\; |b| \ge 100|a|.
\end{equation}

For any $|b| \le |a|/10$, we have that
\[
\prob[t_1 < (10|a|)^{-1}] \le \prob[t_1 < (11|b-a|^{-1}/100)].
\]
Taking $A$ to be sufficiently large to ensure that $(11|b-a|^{-1}/100) < T$ for any $|b| \le |a|/10$, Lemma \ref{lem: nhl}.\ref{lem: nhl c} gives
\begin{equation}\label{eq: neg 4}
    \prob[t_1 < (10|a|)^{-1}] \le C \exp(-c|a|)\;\text{for any }|b| \le |a|/10.
\end{equation}

Note that during Epoch 0 we have $|q|, |u| \le C$ with probability 1, and therefore
\begin{equation}
     \E[\cost_0] \le C |a|^{-1}\;\text{for any}\; b\in \R\label{eq: neg 5.5}
\end{equation}
simply because $t_1 \le (10 |a|)^{-1}$ with probability 1. Moreover,
\[
\E[\cost_0] \le C |b|^{-1} + C\cdot \prob[t_1 > 10|b|^{-1}]\;\text{for any}\; |b| \ge 100|a|,
\]
and therefore \eqref{eq: neg 5} implies that
\begin{equation}\label{eq: ln 4}
\E[\cost_0] \le C |b|^{-1}\;\text{for any}\; |b| \ge 100|a|.
\end{equation}
Combining \eqref{eq: neg 5.5} and \eqref{eq: ln 4} with Lemma \ref{lem: j0} gives
\begin{equation}\label{eq: ln 5}
\E[\cost_0] \le C \cdot \ecost_\op(b)\;\text{for any}\; b \in \R.
\end{equation}

Recall that if we enter Epoch 1 at position $q(t_1) = 2q_0$, then we set $u = 1000\cdot q$ during Epoch 1, and Epoch 1 ends before time $T$ if and only if we encounter $|q| = 4q_0$. By Remark \ref{rmk: cg}.\ref{rmk: cg a}, we have
\begin{equation}\label{eq: ln 6}
    \E[\cost_1(b)\cdot \mathbbm{1}_{q(t_1)=2q_0}] \le C|b|^{-1}\;\text{for any}\; b \le - \frac{|a|}{10}.
\end{equation}
Observe that for any $b \ge |a| / 10$ we have $b^{-1}\ge 100|a+1000\cdot b|^{-1}$. Therefore we can use Lemma \ref{lem: nhl}.\ref{lem: nhl ai} to get
\[
\E[\mathbbm{1}_{q(t_1)=2q_0}\cdot \mathbbm{1}_{t_1 < T - 10b^{-1}}\cdot \E[\mathbbm{1}_{t_2 > t_1+b^{-1}}|\cF_{t_1}]]\le C|b|^{-1}.
\]
Since
\begin{multline*}
\E[\cost_1(b)\cdot \mathbbm{1}_{q(t_1)=2q_0}] \\\le Cb^{-1} + C\E[\mathbbm{1}_{q(t_1)=2q_0}\cdot \mathbbm{1}_{t_1 < T - 10b^{-1}}\cdot \E[\mathbbm{1}_{t_2 >t_1+ b^{-1}}|\cF_{t_1}]],
\end{multline*}
we conclude that
\begin{equation}\label{eq: ln 7}
    \E[\cost_1(b)\cdot \mathbbm{1}_{q(t_1)=2q_0}] \le Cb^{-1}\;\text{for any}\; b \ge \frac{|a|}{10}.
\end{equation}

If we enter Epoch 1 at position $q(t_1) = q_0/2$, then we set $u = -1000q$ during Epoch 1, and Epoch 1 ends before time $T$ if and only if we encounter $|q|= 4q_0$. Proceeding essentially as in the proofs of \eqref{eq: ln 6} and \eqref{eq: ln 7}, we conclude from Remark \ref{rmk: cg}.\ref{rmk: cg a} and Lemma \ref{lem: nhl}.\ref{lem: nhl ai} that
\begin{equation}\label{eq: ln 8}
    \E[\cost_1(b)\cdot \mathbbm{1}_{q(t_1)=q_0/2}] \le C|b|^{-1}\;\text{for any}\; |b| \ge  \frac{|a|}{10}.
\end{equation}

In the event that we enter Epoch 1 at time $t_1 < (10|a|)^{-1}$, the random variable $\cost_1$ is bounded above by a constant with probability 1. We therefore use \eqref{eq: neg 4} to deduce that
\begin{equation}\label{eq: ln 9}
    \E[\cost_1(b) \cdot \mathbbm{1}_{t_1 < (10|a|)^{-1}}] \le C |a|^{-1}\;\text{for any}\; |b| \le \frac{|a|}{10}.
\end{equation}

Recall that if we enter Epoch 1 at time $t_1 =  (10|a|)^{-1}$, then we set $u=0$ until the end of the game at time $T$. By Remark \ref{rmk: cg}.\ref{rmk: cg a}, we then have
\begin{equation}
    \E[\cost_1(b) \cdot \mathbbm{1}_{t_1 = (10|a|)^{-1}} | \cF_{t_1}] \le C |a|^{-1}\;\text{for any}\; b \in\R. 
\end{equation}
In particular,
\begin{equation}\label{eq: ln 10}
    \E[\cost_1(b) \cdot \mathbbm{1}_{t_1 = (10|a|)^{-1}}] \le C |a|^{-1} \;\text{for any}\; |b| \le 100|a|.
\end{equation}
Using \eqref{eq: neg 5}, we see also that
\begin{equation}\label{eq: ln 11}
    \E[\cost_1(b) \cdot \mathbbm{1}_{t_1 = (10|a|)^{-1}}] \le C |b|^{-1} \;\text{for any}\; |b| \ge 100|a|.
\end{equation}
Combining \eqref{eq: ln 6}--\eqref{eq: ln 9}, \eqref{eq: ln 10}, and \eqref{eq: ln 11} with Lemma \ref{lem: j0} gives
\begin{equation}\label{eq: ln 12}
\E[\cost_1] \le C \cdot \ecost_\op(b)\;\text{for any}\; b \in \R.
\end{equation}

Assume that $b \ge 100|a|$. By Lemma \ref{lem: nhl}.\ref{lem: nhl biii}, we then have
\begin{align*}
&\prob[q(t_1) = 2q_0] \le C \exp(-cb), &\\
&\prob[(q(t_1) = q_0/2)\; \text{AND}\; (t_2<T)] \le C \exp(-cb).
\end{align*}
Therefore
\begin{equation}\label{eq: ln bb1}
    \prob[\cE_2] \le C \exp(-cb)\;\text{for any}\; b \ge 100|a|.
\end{equation}

Similarly, if $b \le -100|a|$, then by Lemma \ref{lem: nhl}.\ref{lem: nhl aii} we have
\begin{equation*}
    \prob[q(t_1) = q_0/2] \le C \exp(-c |b|),
\end{equation*}
and by Lemma \ref{lem: nhl}.\ref{lem: nhl biii} we have
\begin{equation*}
     \prob[(q(t_1) = 2q_0) \;\text{AND}\; (t_2 < T)]\le C \exp(-c|b|).
\end{equation*}
Therefore
\begin{equation}\label{eq: ln bb2}
\prob[\cE_2]\le C \exp(-c|b|)\;\text{for any}\; b \le -100|a|.
\end{equation}
If we enter Epoch 2, then we set $u=0$ for the remainder of the game. Therefore, by Remark \ref{rmk: cg}.\ref{rmk: cg a}, we have
\begin{equation}\label{eq: ln b1}
\E[\cost_2(b) | \cF_{t_2}] \le C |a|^{-1}\;\text{for any}\; b \in \R.
\end{equation}
In particular,
\begin{equation}\label{eq: ln 17}
\E[\cost_2 ] \le C |a|^{-1} \;\text{for any}\; |b| \le 100|a|.
\end{equation}
Combining \eqref{eq: ln b1} with \eqref{eq: ln bb1}, \eqref{eq: ln bb2} gives
\begin{equation}\label{eq: ln 18}
    \E[\cost_2 ] \le C |b|^{-1} \;\text{for any}\; |b| \ge 100|a|.
\end{equation}
Combining \eqref{eq: ln 5}, \eqref{eq: ln 12} with \eqref{eq: ln 17}, \eqref{eq: ln 18}, and Lemma \ref{lem: j0} gives
\[
\ecost(\br,b) \le C \cdot \ecost_\op(b)\;\text{for any}\in b \in \R.
\]
Thus, we have succeeded in proving Theorem \ref{thm: main} in the case $a \le - A$.

\bibliographystyle{plain}
\bibliography{ref}

\end{document}